\documentclass[reqno,11pt]{amsart}        
\usepackage{amsfonts,amsmath,amssymb,amsthm,colordvi,mathrsfs,graphicx}
\usepackage{caption,subcaption,float}

\usepackage{amssymb,amsfonts,amsthm,amsmath}
\usepackage{bbm}        
\usepackage{colordvi}
\usepackage{xcolor}
\usepackage{graphicx}
\usepackage{mathrsfs}   
\usepackage{mathtools}

\usepackage[all,knot,poly]{xy}
\usepackage{tikz-cd}
\usepackage{tikz}
\usepackage{verbatim}
\usetikzlibrary{arrows}
\usetikzlibrary{knots}
\tikzset{every path/.style={line width=0.4pt},every node/.style={transform shape,knot crossing,inner sep=1.5pt},>=triangle 60,text node/.style={rectangle,transform shape=false,black}}

 \usetikzlibrary{decorations.markings, arrows.meta}

\theoremstyle{plain}      
\newtheorem{thm}{Theorem}[section]     
\newtheorem{theorem}[thm]{Theorem}     
\newtheorem{corollary}[thm]{Corollary}     
\newtheorem{lemma}[thm]{Lemma}     
\newtheorem{proposition}[thm]{Proposition}

\theoremstyle{remark}      
\newtheorem{example}[thm]{Example} 
 
\newtheorem{remark}[thm]{Remark} 
     
\theoremstyle{definition}      
\newtheorem{definition}[thm]{Definition}     

\newcommand{\supp}{\mathop{\rm supp}\nolimits}

\newcommand{\di}{\displaystyle}

\newcommand{\Vol}{\mathop{\rm Vol}\nolimits}

\newcommand{\C}{{\mathbb C}}

\def\Jac{{\mathop{\rm Jac}}}

\newcommand{\Log}{\mathop{\rm Log}\nolimits}

\newcommand{\Arg}{\mathop{\rm Arg}\nolimits}

\newcommand{\Val}{\mathop{\rm Val}\nolimits}

\title[]{}

 \author{}

\address{Yen-Kheng Lim\\
Department of Physics, Xiamen University Malaysia, Jalan Sunsuria, Bandar Sunsuria, 43900, Sepang, Selangor, Malaysia.}
\email{yenkheng.lim@xmu.edu.my, yenkheng.lim@gmail.com}

\address{Mounir Nisse\\
Department of Mathematics, Xiamen University Malaysia, Jalan Sunsuria, Bandar Sunsuria, 43900, Sepang, Selangor, Malaysia.
}
\email{mounir.nisse@gmail.com, mounir.nisse@xmu.edu.my}
 
\thanks{Y-K. L is supported by Xiamen University Malaysia Research Fund (Grant no. XMUMRF/ 2021-C8/IPHY/0001). MN is supported by Xiamen University Malaysia Research Fund (Grant no. XMUMRF/ 2020-C5/IMAT/0013).}


\usetikzlibrary{
  knots,
  hobby,
  decorations.pathreplacing,
  shapes.geometric,
  calc
}

\tikzset{
  knot diagram/every strand/.append style={
    thick,
    black
  },
  show curve controls/.style={
    postaction=decorate,
    decoration={show path construction,
      curveto code={
        \draw [blue, dashed]
        (\tikzinputsegmentfirst) -- (\tikzinputsegmentsupporta)
        node [at end, draw, solid, red, inner sep=2pt]{};
        \draw [blue, dashed]
        (\tikzinputsegmentsupportb) -- (\tikzinputsegmentlast)
        node [at start, draw, solid, red, inner sep=2pt]{}
        node [at end, fill, blue, ellipse, inner sep=2pt]{}
        ;
      }
    }
  },
  show curve endpoints/.style={
    postaction=decorate,
    decoration={show path construction,
      curveto code={
        \node [fill, blue, ellipse, inner sep=2pt] at (\tikzinputsegmentlast) {}
        ;
      }
    }
  }
}

\begin{document}

{\color{red}{\today}}
\vspace{0.5cm}
\title{Links represented by phases of algebraic curves}

\author{Yen-Kheng Lim and Mounir Nisse}

\begin{abstract}
The prime motivation behind this paper is to prove that any torus link can be realized as the union of the one-dimensional connected components of the set of  critical values of the argument map restricted to a complex algebraic plane curve. Moreover, we give an explicit relation between the Newton polygon of such plane curves,  and the number of components of the given torus link.
This work aims to represent the starting point for a connection between  knot theory, tropical geometry,  and (co)amoebas.

  
 \end{abstract}

\maketitle

\section{Introduction}
Every link $\mathcal{L}$ in the 3-sphere $S^3$ defines a compact, orientable 3-manifold  with tori boundary, that is the link exterior $V(\mathcal{L})= S^3\setminus U(\mathcal{L})$, where $U(\mathcal{L})$ denotes an open neighborhood. Thurston proved that the link complement $S^3\setminus \mathcal{L}$ decomposes into pieces that admit locally homogeneous geometric structures. The most interesting  cases are when  the entire link complements have a hyperbolic structure, namely, a metric of constant curvature $-1$. Moreover, 
 the main problem of knot theory is to determine whether two knots can be rearranged (without cutting) to be exactly similar; more precisely, to be equal or alike. 
The purpose of this paper is to describe a new connection between the set of arguments of complex algebraic curves called 
{\em coamoebas} and {\em knot theory} in order to understand some problems as to find how some invariants of singularities may change by an arbitrarily small deformation, or how the topology and geometry of links can change after a blow up of isolated singularities in the process of resolution. 

The main new idea in this paper is a method  and new view of knots and links as  subsets of phases of complex algebraic plane curves. Thus, our results  open a new way relating deformations of singularities to rearrangements of knots (or links). More precisely, it relates deformations of class of germs of complex algebraic plane curves diffeomorphisms to the rearrangement of links. Moreover, realizing links as phases of complex algebraic curves increases our understanding of the topology and  geometry of 4-dimensional manifolds. Applications of these results on singularities and their deformations and  other techniques will appear in a forthcoming paper in preparation. 


(Co)Amoebas are a very  fascinating notions in mathematics where the  terminology of amoeba has been introduced by I. M. Gelfand, M M. Kapranov and A. V. Zelevinsky in their book (see \cite{GKZ-94}) in 1994, and the terminology of coamoeba has been introduced by M. Passare and A. Tsikh in 2001. Amoebas (resp. coamoebas) have their spines, contours and tentacles (resp. spines, contours and extra-pieces), and they have many applications in real algebraic geometry, complex analysis, mirror symmetry and in several other areas. Amoebas and coamoebas of algebraic hypersurfaces are naturally linked to the geometry  and the combinatorics of  their Newton polytopes. The last fact can be seen in particular with the Viro patchworking principle (i.e., tropical localization) based on the combinatorics of subdivisions of convex lattice polytopes.
 The amoeba $\mathscr{A}$ of an algebraic set $V=\{ f(z)=0 \}$ in the algebraic torus $(\mathbb{C}^*)^n$ is defined as the image of $V$ under the mapping $\Log : (z_1,\ldots , z_n)\mapsto (\log |z_1| ,\ldots , \log |z_n| )$. 
The coamoeba $co\mathscr{A}$ of an algebraic set $V=\{ f(z)=0\}$ in $(\mathbb{C}^*)^n$ is defined as its image under the argument mapping $\Arg : (z_1,\ldots ,z_n)\mapsto (e^{i\arg (z_1)},\ldots ,e^{i\arg (z_n)})$.  For those unfamiliar with tropical geometry and, in particular, with the notions of (co)amoebas, the critical values, and the critical points of the logarithmic and the argument maps, we recall these notions in sections 2 and 3.


In this paper, we show  that for any torus knot $\mathcal{K}$  there exists a complex  algebraic plane curve such that  the one-dimensional  connected  component of the contour (i.e. critical values of the argument map) of its coamoeba realizes  $\mathcal{K}$. 
More generally, we show that any torus link $\mathcal{L}$ is the union of the one-dimensional connected components of the critical values of the argument map restricted to a complex algebraic plane curve; see Theorem \ref{maintheorem}.
Moreover, we show that torus links are naturally in connection with the geometry and the combinatorics of the Newton polygons of those algebraic curves.
Also, we give some examples of coamoebas of  complex algebraic plane curves containing torus links as the set of the critical values of the argument map restricted to these curves.



\vspace{0.3cm}

\noindent {\bf Organization of the paper.} This paper contains  new results announced by the second author on February 2023 in the Geometry Seminar at Texas A\&M University, and in the Geometry/Topology Seminar at UC. Davis. 
The remainder of this paper is organized as follows. Section 2 contains preliminary materials, in particular we give an overview of a number of 
properties of  complex hypersurfaces amoebas proved in  \cite{NS1-13}. 
 Moeover, it reviews some properties of complex hypersurfaces coamoebas.
In Section 3, we review some properties of the set of critical values of the logarithmic map and the argument map already shown in \cite{M1-00} and   \cite{MN-15}. In Section 4,  we describe torus knots and torus links as critical values of the argument map restricted to some special complex algebraic plane curves. In Section 5, some known torus knots and torus links, as {\em Borromean rings}, {\em Hopf link}, and the {\em Trefoil knot}, 
are realized as a particular subsets in the coamoebas of complex algebraic plane curves, and  we give some other examples.





\section{Preliminaries}
Let $V$ be  an algebraic hypersurface in the
complex algebraic torus $(\mathbb{C}^*)^n$, where $\mathbb{C}^*
=\mathbb{C}\setminus \{ 0\}$, and $n\geq 1$ is an integer. This means that $V$ is the zero locus of a polynomial:
$$
f(z) = \sum_{\alpha\in \supp (f)} a_{\alpha}z^{\alpha}, \quad\quad\quad\quad
z^{\alpha}=z_1^{\alpha_1}z_2^{\alpha_2}\ldots z_n^{\alpha_n}\quad\quad\quad\quad\quad\quad\quad\quad  
$$
where each $a_{\alpha}$ is a non-zero complex number and $\supp (f)$ is a
finite subset of $\mathbb{Z}^n$, called the {\em support} of the polynomial
$f$, with convex hull, in $\mathbb{R}^n$, the {\em Newton polytope}
$\Delta_f$ of $f$.
Moreover, we assume that $\supp (f)\subset
\mathbb{N}^n$ and $f$ has no factor of the form $z^{\alpha}$ with
$\alpha =(\alpha_1,\ldots , \alpha_n)$.


The {\em amoeba}  $\mathscr{A}_f$ of an algebraic hypersurface $V_f\subset
 (\mathbb{C}^*)^n$, or more generally any subset of the complex algebraic torus, 
 is by definition the image of $V_f$ under the map (see M. Gelfand, M.M. Kapranov
 and A.V. Zelevinsky \cite{GKZ-94}):
\[
\begin{array}{ccccl}
\Log&:&(\mathbb{C}^*)^n&\longrightarrow&\mathbb{R}^n\\
&&(z_1,\ldots ,z_n)&\longmapsto&(\log | z_1| ,\ldots ,\log |
z_n|).
\end{array}
\]

\vspace{0.2cm}

The argument map is the map  defined as follows:
\[
\begin{array}{ccccl}
\Arg&:&(\mathbb{C}^*)^n&\longrightarrow&(S^1)^n\\
&&(z_1,\ldots ,z_n)&\longmapsto&(\arg (z_1),\ldots ,\arg (
z_n) ).
\end{array}
\]
where $\di\arg (z_j) = \frac{z_j}{|z_j|}$. The {\em coamoeba} of $V$,  denoted by $co\mathscr{A}$, is its image 
under the argument map (defined for the first time in a talk given  by Passare in 2004). By the same abuse of notation, we denote the restriction of $\Arg$ and $\Log$ to each variety $V\subset (\mathbb{C}^*)^n$ by the same symbol. 

\subsection{Complex tropical hypersurfaces with a simplex Newton polytope}  ${}$ 

\par
 This section reviews essential background on (co)amoebas that will be used in this paper.
Let $\mathbb{K}$ be the field of the Puiseux series with real power, which is the field of formal series $\displaystyle{a(t) = \sum_{j\in A_a}\xi_jt^j}$ with $\xi_j\in \mathbb{C}^*$ and $A_a\subset \mathbb{R}$  a well-ordered subset  of $\mathbb{R}$, which means that any subset  of $A_a$ has a smallest element. It is well known that the field $\mathbb{K}$ is an algebraically closed field of characteristic zero. This field is called  the field of generalized Puiseux series. 
In the following, $K$ will either be the field complex number $\mathbb{C}$ or the field of generalized Puiseux series $\mathbb{K}$.
Let $a=(a_1,\ldots ,a_n)\in (K^*)^n $ and  $P_a\subset(K^*)^n$ be the hyperplane defined by the polynomial $f_a(z_1,\ldots ,z_n) = 1+\sum_{j=1}^na_jz_j$, and  let $P_1$ be the hyperplane defined by $f_1(z_1,\ldots ,z_n) =1+\sum_{j=1}^nz_j $, then it is clear that 
if $\tau_{a^{-1}}$ is the translation in the multiplicative group $(K^*)^n$ by $a^{-1}$, then  we have
$P_a = \tau_{a^{-1}}(P_{1})$. Let $L$ be an invertible matrix with integer coefficients and positive determinant such that:
$$
{}^tL = \left(
\begin{array}{ccc}
\alpha_{11}&\hdots&\alpha_{1n}\\
\vdots&\ddots&\vdots\\
\alpha_{n1}&\hdots&\alpha_{nn}
\end{array}
\right) ,
$$
and $\Phi_{L,\, a}$ be the homomorphism of the algebraic torus defined as follow.
\[
\begin{array}{ccccl}
\Phi_{L,\, a}&:&(K^*)^n&\longrightarrow&(K^*)^n\\
&&(z_1,\ldots ,z_n)&\longmapsto&(a_1\prod_{j=1}^nz_j^{\alpha_{1j}},\ldots ,a_n\prod_{j=1}^nz_j^{\alpha_{nj}}).
\end{array}
\]
Let $V_f\subset (\mathbb{K}^*)^n$ be the hypersurface defined by the polynomial 
$$
f(z_1,\ldots ,z_n)=1+\sum_{k=1}^na_k\prod_{j=1}^nz_j^{\alpha_{kj}}.
$$
Let us denote by 
$co\mathscr{A}(V_f)$ 
the set of arguments of $V_f$, and by the same notation  $co\mathscr{A}(V_f)$ its lifting in the universal covering $\mathbb{R}^n$ of the real torus (abuse of notation). Let $\tau_{\Arg (a)}$ be the translation by the vector $(\arg (a_1),\ldots ,\arg (a_n))$ in $\mathbb{R}^n$.

If we denote by $\tau_{\tilde{a}}$ the translation in $\mathbb{R}^n$ by the vector ${}^tL^{-1}(\arg (a_1),\ldots ,\arg (a_n))$, then we have the following commutative diagram:
\begin{equation}
\xymatrix{
V_f\ar[d]_{\Arg}\ar[rr]^{\Phi_{L,\, a}}&&P_1\ar[d]^{\Arg}\cr
\mathbb{R}^n\ar[rr]^{{}^tL\circ\tau_{\tilde{a}}}&&\mathbb{R}^n.
}\nonumber
\end{equation}
We have the same diagram if we replace $\Arg$ by the logarithmic map $\Log$
  (or the valuation map $\Val$ if we work in $(\mathbb{K}^*)^n$).

\vspace{0.2cm}
\newpage

\begin{lemma}\label{lem1} Let $V_f$ be a hypersurface defined by a polynomial $f$ with Newton polytope a simplex $\Delta$, and such that its support $\supp (f)$ is precisely the vertices of $\Delta$. Then the following hold:
\begin{itemize}
\item[(i)]\, if $K = \mathbb{C}$ (resp. $\mathbb{K}$), then the amoeba $\mathscr{A}_f$ of the hypersurface $V_f$ is the image under $\tau^{-1}_{\Log (a)}\circ {}^tL^{-1}$ (resp. $\tau^{-1}_{\Val (a)}\circ{}^tL^{-1}$) of the amoeba of the standard hyperplane $P_1$. 
\item[(ii)]\, if $K = \mathbb{C}$  or $\mathbb{K}$, then the coamoeba $co\mathscr{A}_f$ of the hypersurface $V_f$ is the image under $\tau^{-1}_{\Arg (a)}\circ{}^tL^{-1}$ of the coamoeba of the standard hyperplane $P_1$. In particular, the number of its complement components in the real torus $(S^1)^n$ is equal to $n!\Vol (\Delta)$.
\end{itemize}
\end{lemma}
This means that we have the following:
$$
\hspace{0.3cm} \tau^{-1}_{\Log (a)}\circ{}^tL^{-1}\ (\mathscr{A}_{P_1}) = \mathscr{A}_f, \quad\mbox{and}\quad
\tau^{-1}_{\Arg (a)}\circ {}^tL^{-1}  (co\mathscr{A}_{P_1}) = co\mathscr{A}_f. 
$$

\begin{proof} 
First of all, we can see that the Newton polytope $\Delta_f$ of $f$ is the image under the linear map $L$ of the standard simplex. The matrix $L$ is invertible, so $\Phi_{L,\, a}(V_f) = P_1$. Indeed, if $(z_1',\ldots ,z_n')$ is in $P_1$, then there exists $(u_1,\ldots ,u_n)\in \mathbb{C}^n$ such that for any $1\leq j\leq n$, we have $z_j'= a_je^{u_j}$. The matrix $L$ is invertible, so its column vectors $\alpha_k$ are linearly independent. Hence, there exists a vector $(v_1,\ldots ,v_n)\in \mathbb{C}^n$ which is a solution of the following linear system:
$$
\begin{array}{ccccccc}
\alpha_{11}x_1&+&\hdots&+&\alpha_{1n}x_n&=& u_1\\
\alpha_{21}x_1&+&\hdots&+&\alpha_{2n}x_n&=& u_2\\
\vdots&\vdots&\vdots&\vdots&\vdots&\vdots&\vdots\\
\alpha_{n1}x_1&+&\hdots&+&\alpha_{nn}x_n&=& u_n;
\end{array}
$$
and then 
$$
\begin{array}{lll}
\Phi_{L,\, a} (e^{v_1},\ldots ,e^{v_n}) &= & (a_1\prod_j e^{\alpha_{1j}v_j},\ldots ,a_n\prod_j e^{\alpha_{nj}v_j})\\
&=& (a_1e^{u_1},\ldots ,a_ne^{u_n})\\
&=& (z_1',\ldots ,z_n').
\end{array}
$$
But  $(e^{v_1},\ldots ,e^{v_n})\in V_f$, because $1+ \sum_k a_k\prod_j (e^{v_j})^{\alpha_{kj}} = 1 + \sum_k a_k e^{\sum_jv_j\alpha_{kj}}  = 1 + \sum_k a_k e^{u_k} = 1 + \sum_k z_k' = 0$. So, $\Phi_{L,\, a}(V_f) = P_1$, and the Lemma \ref{lem1} is done after using the properties of the logarithmic and the argument maps on one hand, and the properties of the amoeba and the coamoeba of the standard hyperplane on the other hand, and the fact that $\det (L) = n!\Vol (\Delta )$. 

\end{proof}

\begin{example}

Using Lemma \ref{lem1}, we draw in Figure 1 the coamoeba  of the complex  curve defined by the polynomial $f_1(z_1,z_2)= 1+ z_1^2z_2^3+z_1^3z_2$
where the matrix  ${}^tL_1^{-1}$ is equal to $\frac{1}{7}\left(
  \begin{array}{cc} 3&-1\\ -2&3\end{array}\right) $.


\begin{figure}[H]
\centering
\begin{subfigure}[b]{0.49\textwidth}
    \centering
    \begin{tikzpicture}[scale=0.13]
        \draw (-12,-12) grid (12,12);

        \draw[line width=0.4mm,green] (-12,0) -- (-8,-12);
        \draw[line width=0.4mm,green] (0,-12) -- (-8,12);
        \draw[line width=0.4mm,green] (0,12) -- (8,-12);
        \draw[line width=0.4mm,green] (8,12) -- (12,0);

        \draw[line width=0.4mm,red] (-12,-8) -- (-6,-12);
        \draw[line width=0.4mm,red] (-6,12) -- (12,0);
        \draw[line width=0.4mm,red] (-6,12) -- (12,0);
        \draw[line width=0.4mm,red] (-12,0) -- (6,-12);
        \draw[line width=0.4mm,red] (6,12) -- (12,8);
        \draw[line width=0.4mm,red] (-12,8) -- (12,-8);
  
        \draw[line width=0.4mm,blue] (0,-12) -- (12,-6);
        \draw[line width=0.4mm,blue] (-12,-6) -- (12,6);
        \draw[line width=0.4mm,blue] (-12,6) -- (0,12);

        \shade[top color=yellow, bottom color=black] (-12,-6) -- (-12,-8) -- (-8.6666,-10.3333) -- (-10.33,-5.3333) -- cycle;
        \shade[top color=yellow, bottom color=black] (-12,0) -- (-10.3333,-5.3333) -- (-6.86,-3.5) -- cycle;
        \shade[top color=yellow, bottom color=black] (-3.5,-1.8) -- (-1.8,-6.9) -- (-6.86,-3.5) -- cycle;
        \shade[top color=yellow, bottom color=black] (-3.5,-1.8) -- (0,0) -- (-5.3,3.5) -- cycle;
        \shade[top color=yellow, bottom color=black] (5.1,-3.4) -- (0,0) -- (3.5,1.8) -- cycle;
        \shade[top color=yellow, bottom color=black] (6.9,3.5) -- (1.8,6.9) -- (3.5,1.8) -- cycle;
        \shade[top color=yellow, bottom color=black] (6.9,3.5) -- (10.3,5.1) -- (12,0) -- cycle;
        \shade[top color=yellow, bottom color=black] (8.6,10.3) -- (10.3,5.1) -- (12,6)  -- (12,8) -- cycle;
        \shade[top color=yellow, bottom color=black] (8.6,10.3) -- (8,12) -- (6,12) -- cycle;
        \shade[top color=yellow, bottom color=black] (0,12) -- (1.8,6.9) -- (-3.4,10.3) -- cycle;
        \shade[top color=yellow, bottom color=black] (-10.4,6.9) -- (-6.9,8.56) -- (-5.3,3.5) -- cycle;
        \shade[top color=yellow, bottom color=black] (-8,12) --  (-6,12) -- (-3.4,10.3) -- (-6.9,8.56) -- cycle;
        \shade[top color=yellow, bottom color=black] (-10.4,6.9) -- (-12,8) -- (-12,6) -- cycle;
        \shade[top color=yellow, bottom color=black] (-8,-12) -- (-8.6666,-10.3333) -- (-6,-12) -- cycle;
        \shade[top color=yellow, bottom color=black] (0,-12) -- (-1.8,-6.9) -- (3.4,-10.3) -- cycle;
        \shade[top color=yellow, bottom color=black] (6,-12) -- (8,-12) -- (6.86,-8.7) -- (3.4,-10.3) -- cycle;
        \shade[top color=yellow, bottom color=black] (10.3,-6.9) -- (12,-6) -- (12,-8) -- cycle;
        \shade[top color=yellow, bottom color=black] (5.1,-3.4) -- (6.87,-8.7) -- (10.3,-6.9) -- cycle;
\end{tikzpicture}
    \caption{}
    \label{fig_M1a}
  \end{subfigure}
  \begin{subfigure}[b]{0.49\textwidth}
    \centering
    \begin{tikzpicture}[scale=0.25]
        \draw (-6,-6) grid (6,6);
        \draw[line width=0.3mm,green] (-6,-4) -- (6,0);
        \draw[line width=0.3mm,green] (-6,0) -- (6,4);
        \draw[line width=0.3mm,green] (-6,4) -- (0,6);
        \draw[line width=0.3mm,green] (0,-6) -- (6,-4);
        \draw[line width=0.3mm,red] (-6,-3) -- (0,0);
        \draw[line width=0.3mm,red] (0,0) -- (6,3);
        \draw[line width=0.3mm,red] (-6,3) -- (0,6);
        \draw[line width=0.3mm,red] (0,-6) -- (6,-3);
        \draw[line width=0.3mm,blue] (-6,0) -- (0,0);
        \draw[line width=0.3mm,blue] (0,0) -- (6,0);

        \shade[top color=yellow, bottom color=black] (0,0) -- (6,3) -- (6,4) -- (-6,0) -- cycle;
        \shade[top color=yellow, bottom color=black] (0,0) -- (6,0) -- (-6,-4) -- (-6,-3) -- cycle;
        \shade[top color=yellow, bottom color=black] (0,-6) -- (6,-4) -- (6,-3) -- cycle;
        \shade[top color=yellow, bottom color=black] (-6,3) -- (0,6) -- (-6,4) -- cycle;
        %
\end{tikzpicture}    
    \caption{}
    \label{fig_M1b}
  \end{subfigure}
\caption{In these figues the squares represents a fundamental domain $[0,2\pi]\times [0,2\pi]$ of the real torus $S^1\times S^1$.  (A)  The coamoeba of the curve defined by the polynomial $f_1$. (B) The coamoeba of the curve with defining polynomial $g(z,w)=z+w^2+w^3$.} 
\label{fig:M}
\end{figure}


\end{example}


More generally, if $L$ is an invertible $(n\times n)$-matrix with integer entries, and 
$f(z)= \sum a_{\alpha\in A} z^{\alpha}$ is a Laurent polynomial,  then $f_L$ is the polynomial defined as follows:
 $$
 f_L(z):=\sum_{\alpha\in A} a_\alpha z^{L(\alpha)},
 $$
 where $\alpha\in \mathbb{Z}^n$ is viewed as a column vector. In this case, the map  $x\mapsto {}^tL(x)$ where ${}^tL$ denotes the transpose matrix, sends the amoeba (resp. coamoeba) of $V_{f_L}$ into the amoeba (resp. coamoeba) of $V_f$. For more details see \cite{NS1-13}.

\vspace{0.2cm}

\section{Critical points of the logarithmic and the argument maps}

In this section, we will explore the relationship between the sets of critical points of the logarithmic and the argument maps, and  we
give an explicit description of these sets. More precisely, given a map  $f$, we denote by $\Jac(f)$ the Jacobian matrix of $f$.
It was shown in  \cite{MN-15} the  following proposition:

\begin{proposition}[Madani-Nisse, 2015]\label{jacobian matrix}
Let $V$ be a $k$-dimensional complex submanifold in $(\C^*)^n$. The maps $\Log$ and $\Arg$ are well defined on $V$ and  
\begin{equation}
 \partial \Log=\frac{1}{\Arg}\partial\Arg,\quad   \bar\partial \Log=\frac{-1}{\Arg}\bar\partial\Arg .\nonumber
\end{equation} 
\end{proposition}

\begin{proof}
We denote by $\{z_j\}_{1\leq j\leq n}$ the complex coordinates on $ \C^n$ and by $\{t_j\}_{1\leq j\leq k}$ the complex coordinates on $V$ given by a local chart $(\Omega,f)$ (i.e. $\forall z\in\Omega,\; t_j=f_j(z)$), where $\Omega$ is an open set of $V$ and $f$ is a holomorphic function from an open set of $\C^n$ to $\C^k$. Since $V$ is a complex submanifold of $(\C^*)^n$, the injection map $\imath: V\hookrightarrow (\C^*)^n$ is holomorphic. 
By definition,  for any $z\in V$ we have $\imath (z)=e^{Log z}\Arg z$. Since $\imath$ is holomorphic, $\bar\partial \imath (z)=0$ for any $z\in V$ (i.e. $\forall j\leq k,\; \partial_{\bar t_j} \imath (z)=0$ ). It implies that for any $j=1,\ldots,k$ and $z\in \Omega$ we have
\begin{gather*}
\partial_{\bar t_j} \Log (z) =-\frac{1}{\Arg(z)}\partial_{\bar t_j} \Arg(z), \quad\partial_{t_j}\Log(z) =\frac{1}{\Arg(z)} \partial_{t_j}\Arg(z),
\end{gather*}
where the second equality holds by conjugating the first one. The statement of the proposition follows.
\end{proof}


\vspace{0.2cm}

From  Proposition \ref{jacobian matrix}, we obtain the following:
  
\begin{corollary}\label{jacobian det}
Let $V$ be a $k$-dimensional complex submanifold in $(\C^*)^n$. The set of critical points of the maps $\Log$ and $\Arg$ coincide.  
\end{corollary}

\begin{proof}
Let 
\begin{equation*}
S=\{z\in V\; \vert\; \textrm{rank}\;\Jac\,\Log_z <\min(n,2k)\}
\end{equation*}
be the set of critical points of the logarithmic map $\Log$. Using Proposition \ref{jacobian matrix}, we conclude that  the Jacobian matrices of $\Log$ and $\Arg$ have the same rank at any point in $V$. Thus, the set of critical points of $\Arg$ is also $S$.

\end{proof}
\vspace{0.2cm}

Let  $V\subset (\mathbb{C}^*)^n$ be  a complex algebraic hypersurface
defined by a polynomial $f$ and suppose that $V$ is nowhere singular. The {\em logarithmic Gauss map} 
is a rational map from  $V$ to $\mathbb{CP}^{n-1}$ defined as follows:
 \begin{eqnarray*}
\gamma :&V&\longrightarrow\,\,\,\mathbb{CP}^{n-1}\\
&z&\longmapsto\,\,\, \gamma (z) =  [z_1\frac{\partial f}{\partial z_1}(z):\cdots :z_n\frac{\partial f}{\partial z_n}(z)],
\end{eqnarray*}

\noindent We have the following  commutative diagram:

\begin{equation}
\xymatrix{
V \ar[rr]^{\gamma}&& \mathbb{C}\mathbb{P}^{n-1}\cr
Critp(\Log_{|V})\ar[rr]^{\gamma_c}\ar[u]^{\cup}\ar[dr]_{\Log}&&\mathbb{R}\mathbb{P}^{n-1}\ar[u]_{\cup} \cr
&Critv(\Log_{|V})\ar[ru]_{g},
}\nonumber
\end{equation}
where $\cup$  denotes the natural inclusion. The map 
$g$ is the usual Gauss map defined on the smooth part of 
$Critv(\Log_{|V})$,  and $\gamma_c = \gamma_{|Critp(\Log_{|V})}$ is the restriction of $\gamma$ to the set of  critical points $Critp(\Log_{|V})$ of the logarithmic map. This diagram was used in several places e.g.  by Mikhalkin in several  talks,  in Nisse's thesis,  and in  Passare and Risler's  paper \cite{PR-11}.

\vspace{0.2cm}

\begin{lemma}[Mikhalkin, 2000]
The set of critical points of the logarithmic map ${Critp(\Log_{|V})}$ is equal to the inverse image of the real projective space under the logarithmic Gauss map $\gamma$.

\end{lemma}

\vspace{0.2cm}

\begin{definition}\label{regular-point}
A point  $x$  in $Critv(\Log_{|V})$ is called regular if and only if  $\Log^{-1}(x)\cap V$ is contained in the set of regular points of the restriction of the logarithmic Gauss map to $Critp(\Log_{|V})$. In other  words,  $\Log^{-1}(x)\cap V$ contains no critical point of $\gamma$.
\end{definition}

\vspace{0.2cm}

\begin{lemma}\label{lemmaC}
Let $V$ be a complex algebraic hypersurface. Let $x$ be a point in the boundary of the amoeba $\partial\mathscr{A}(V)$. then the set $(\Log^{-1}(x)\cap V)$ is contained in $Critp(\Log_{|V})$.
\end{lemma}

\begin{proof}
Assume  that there exists $z\in V$ such that $\Log (z) = x$ and  $z$ is a regular point of the logarithmic map, i.e., the Jacobian $\Jac (\Log_{|V})_{z}$ of the logarithm map restricted to $V$ at the point $z$ has maximal rank.
The fact that the set of regular points of the logarithmic map is an open subset of $V$, implies that there exists an open subset $U_z$ in $V$ containing $z$, such that $\Log_{|U_z}$ is a submersion. Hence, the point 
$x$ must be a regular value. In other words, the point $x$ must be in the interior of the amoeba 
and  not in its boundary. This contradict our  hypothesis on $x$.
\end{proof}

 \vspace{0.0cm}
 
\begin{example} 
Let $\mathcal{H}$ be the real algebraic plane curve (hyperbola) parametrized as follows:
\begin{eqnarray*}
\rho :&\mathbb{C}^*\setminus \{ -1, -\frac{1}{6}\}&\longrightarrow\,\,\,(\mathbb{C}^*)^{2}\\
&z&\longmapsto\,\,\, \rho (z) =  -\frac {z+\frac{1}{6}}{z+1}.
\end{eqnarray*}
	Its amoeba has a non regular critical value $x_0$, called by Mikhalkin a pinching point (Remark 10,  \cite{M1-00}). The inverse image of the point $\di x_0 = (-\frac{\log 6}{2}, -\frac{\log 6}{2})$  by the logarithmic map in $\mathcal{H}$ is a circle 
	but not geodesic in the flat torus $\Log^{-1}(x_0) = S^1\times S^1$.  As the set of critical points of the logarithmic map and the argument map coincide, then the image of
 $\mathcal{H}\cap\Log^{-1}(x_0)$  under the argument map is a circle which is equal to the 1-dimensional connected component of the set of critical values of the argument map restricted to
 $\mathcal{H}$.  Moreover, the restriction of the argument map to $\mathcal{H}\cap\Log^{-1}(x_0)$ is an homeomorphism into its image.
 
  \vspace{0.1cm}
 
 The set of critical values of the argument map is a non geodesic circle $\delta$ which has two different real points (i.e., it intersects the finite real subgroup $(\mathbb{Z}_2)^2$ of  the whole real torus in two points) union the isolated point $(\pi , \pi)$. The circle $\delta$ is also critical for the logarithmic Gauss map $\gamma_c$. More precisely, there are two real branches $\mathscr{B}_1$ and $\mathscr{B}_2$ of critical points intersecting $\mathcal{H}$ in two different real points contained in two different quadrants of $(\mathbb{R}^*)^2$, namely the quadrants $(+,-)$ and $(-,+)$. The image of each branch by the logarithmic map has an inflection point at $x_0$.

\begin{figure}[H]
 \centering 
  \begin{subfigure}[b]{0.4\textwidth}
    \centering
    \includegraphics[width=0.7\textwidth]{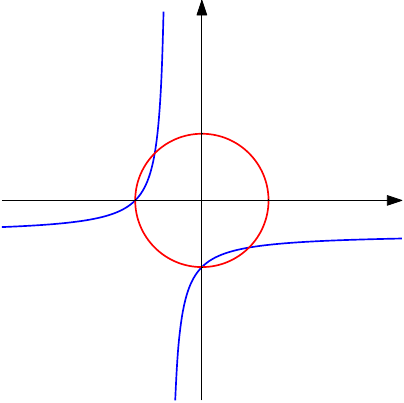}
    \caption{}
    \label{fig_Pinch-Amoeba-A}
  \end{subfigure}
  \begin{subfigure}[b]{0.5\textwidth}
    \centering
    \includegraphics[width=1.3\textwidth]{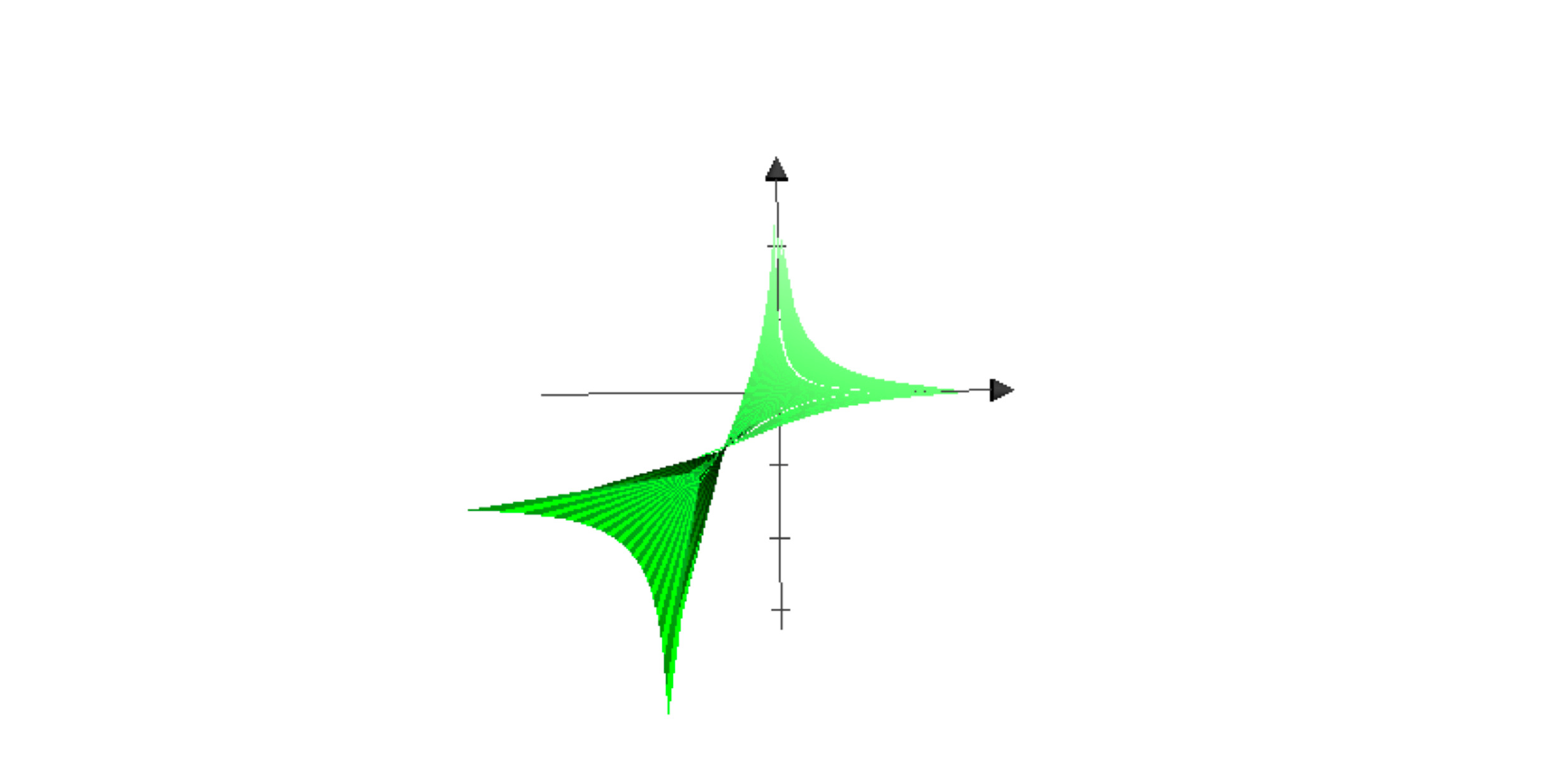}
    \caption{}
    \label{fig_Pinch-Amoeba-AR-1}
  \end{subfigure}
  \caption{(A) The real part of the hyperbola with defining polynomial $f(z,w) = \frac{1}{6} +z+w+zw$ and a circle of radius $\frac{1}{6}$. (B) The amoeba   of the real hyperbola in $(\mathbb{C}^*)^2$ with defining polynomial $f(z,w) = \frac{1}{6} +z+w+zw$.}
  \label{fig_Pinch}
\end{figure}

 In Figure 2. (A), the image under the logarithmic map of the two points intersection of the red circle with the real part of the hyperbola outside the axis is precisely the point $x_0$.

 \vspace{0.1cm}
 
 \begin{remark}
 The inverse image of the point ${x_0}\in \mathbb{R}^2$ by the logarithmic map intersects  the real algebraic plane curve $\mathcal{H}$ into a circle $\delta\subset \Log^{-1}(x_0) = S^1\times S^1$ which is a subset of the set of critical points of the logarithmic map. On the other hand, the circle  $\delta$ is homeomorphic to its image $\Arg (\delta)$ under the argument map  which is a subset of the set of critical values  of the argument map. For this reason we say that it can be represented  as a subset of the critical points of the argument map (which is equal to the set of critical points of the logarithmic map by Corollary 3.2) and in the same time as a subset of the set of critical values of the argument map.
 \end{remark}

\begin{figure}[H]
\begin{center}
 \centering 
  \begin{tikzpicture}[scale=0.25]
        \draw (-6,-6) grid (6,6);

        \draw[line width=0.4mm,green] (-6,0) -- (6,0);
        \draw[line width=0.4mm,black] (0,-6) -- (0,6);

        \filldraw [gray] (0,0) circle (4pt);
        \draw[line width=0.7mm,red]  (6,0) .. controls (0.5,-0.5) .. (0,-6);  
        \draw[line width=0.7mm,red]  (-6,0) .. controls (-0.5,0.5) .. (0,6);
        \shade[top color=yellow, bottom color=black] (0,0) -- (-6,0) .. controls (-0.5,0.5) .. (0,6) -- cycle;
        \shade[top color=yellow, bottom color=black] (0,0) -- (0,-6) .. controls (0.5,-0.5) .. (6,0) -- cycle;
    \end{tikzpicture}
    \caption{The coamoeba of the real hyperbola in a fundamental domain  $[0,2\pi]\times [0, 2\pi]$ of the 2-dimensional real torus $S^1\times S^1$. The critical values of the argument map is the topological circle in red color (which isotopic to the diagonal).}
    \label{c}
\end{center}
\end{figure} 

\end{example}

\section{(Co)amoebas and torus knots}


Throughout this section, we consider the set of  knots and links in  the $3$-dimensional sphere $S^3$. As $S^3$ can be viewed as $\mathbb{R}^3 \cup \{\infty\}$,  and  since any $1$-dimensional submanifold in $S^3$ may avoid a point,  then up to equivalence, we can consider the set of  knots and links as subsets in $\mathbb{R}^3$.


\begin{definition}
A {\em knot} is a   subset of $\mathbb{R}^3$ homeomorphic to  a circle.  
\end{definition}


\begin{definition}
A {\em link} with $k$-components is the image of an embedding of a disjoint union of $k$ copies of $S^1$ in $\mathbb{R}^3$. A link with a single component is called a {\em knot}. Equivalently a link is a subset of $\mathbb{R}^3$ homeomorphic to a disjoint union of circles.
\end{definition}


A knot is {\em trivial}, or {\em unknot} if it is not knotted.
 

\begin{definition}
A {\em torus knot}  (resp. {\em torus link}) is a knot  (resp. a link) that lies on  an unknotted $2$-dimensional real torus embedded in the $3$-dimensional real space $\mathbb{R}^3$. 
\end{definition}


Every torus knot (resp. torus link) is associated to a pair of coprime (resp. not coprime) integers $p$ and $q$.  In this case,  the number of connected components  of the torus link is the greatest common divisor $\textrm{gcd}(p, q)$. A torus knot is trivial or unknot if and only if either $p$ or $q$ is equal to $1$ or $-1$.

\begin{definition}
Let $I=[A,B] \subset \mathbb{R}^n$ be a segment  with $A$, and $B$ in $\mathbb{Z}^n$. The integer length of the  segment $I$, denoted  by $\textrm{intl}_\mathbb{Z}(I)$,  is defined as follow:
$$
\textrm{intl}_\mathbb{Z}(I) :=\#\big(\mathbb{Z}^n\cap I\big) -1,
$$
where $\# (X)$ denotesthe cardinality of the set $X$.
 \end{definition}

\vspace{0.1cm}

Let us denote by $\mathscr{P}(2)$ the set of polynomials  whose Newton polygons are parallelograms, and consider  the set of real algebraic plane curves  with defining  polynomials in $\mathscr{P}(2)$.
   Since we are working in the complex algebraic torus $(\mathbb{C}^*)^2$, without loss of generality, a polynomial with Newton polygon a square $\mathscr{S}$ of sides of length $1$ can be written as follows:
$$
f(z, w) = 1 + \tau_1 z + \tau_2 w + \tau_3zw,
$$ 
where $\tau_i$ are real numbers different than zero for $i=1,2,3$. Thus, if $g$ is a polynomial in $\mathscr{P}(2)$, then its Newton polygon is, up to a translation by a vector with integer components,  the image of $\mathscr{S}$ under a  linear map in $GL(2,\mathbb{Z})$.
In other words, $g = f_L$, where 
$$
L = \left(\begin{array}{cc} a&b\\ c&d\end{array}\right),
$$
with $a, b, c, d$ are integers, and $g(z,w) = 1 + \tau_1 z^aw^c + \tau_2 z^bw^d + \tau_3  z^{a+b}w^{c+d}$.
 Therefore, using Lemma \ref{lem1}, the amoeba $\mathscr{A}(V_g)$ (respectively the coamoeba $co\mathscr{A}(V_g)$) of $V_g=\{  (z,w)\in (\mathbb{C}^*)^2 \,\big|\,  g(z,w)=0\}$ is the image  under the linear map $^{t}L^{-1}$ of the amoeba $\mathscr{A}(V_f)$ (respectively coamoeba $co\mathscr{A}(V_f)$) of the hyperbola defined by the polynomial $f$.

 \vspace{0.2cm}
 
 
 We begin by summarising our main results of this section:

\begin{theorem}\label{maintheorem}
Let $\mathcal{C}_L$ be the real plane curve with defining polynomial $g$ as above with $\tau_1=\tau_2=\tau$, and $\tau_3=1$,  where $\tau$ is a real number such that $\tau\ne \pm1$.
 Then, the amoeba $\mathscr{A}_{\mathcal{C}_L}$ of the curve $\mathcal{C}_L$  contains the origin  $\mathcal{O}\in \mathbb{R}^2$, and the number of connected components $s$ of
  the inverse image of the origin  of $\mathbb{R}^2$ with the logarithmic map is as follows:
\begin{itemize}
\item[(i)]\, If  $|\tau|>1$, then $s$ is equal to the integer length  of the diagonal $\Big[\left(\begin{array}{c} a\\ c\end{array}\right), \left(\begin{array}{c} b\\ d\end{array}\right)\Big]$ of the Newton polygon of $g$,  and   $\Log^{-1}(\mathcal{O})$ is  a link of $s$ torus knots. Moreover, 
$$ 
s = \textrm{intl}_\mathbb{Z}(\Big[\left(\begin{array}{c} a\\ c\end{array}\right), \left(\begin{array}{c} b\\ d\end{array}\right)\Big])= \Big|gcd(b-a,d-c)\Big|,
$$
\item[(ii)]\, if $0<|\tau|<1$, then $s$ is equal to the integer length  of the second diagonal\\
$\Big[\left(\begin{array}{c} 0\\ 0\end{array}\right), \left(\begin{array}{c} a+b\\ c+d\end{array}\right)\Big]$ of the Newton polygon of $g$,  and   $\Log^{-1}(\mathcal{O})$ is  a link of $s$ torus knots,  where $s$ is as follows:
$$ 
s = \textrm{intl}_\mathbb{Z}(\Big[\left(\begin{array}{c} a+b\\ 0\end{array}\right), \left(\begin{array}{c} 0\\ c+d \end{array}\right)\Big])= \Big|gcd(b+a,c+d)\Big|.
$$
\end{itemize}
\end{theorem}

\vspace{0.1cm}

\begin{proposition}
Let $\mathcal{C}_f$ be the real plane curve with defining polynomial $f$ as above with $\tau_1=\tau_2=\tau$, and $\tau_3=1$,  where $\tau$ is a  real number such that $\tau\ne \pm 1$ (i.e. the Newton polygon of $f$ is the unit square).
 Then, the amoeba $\mathscr{A}_{\mathcal{C}_f}$ of the curve $\mathcal{C}_f$ contains the origin  $\mathcal{O}\in \mathbb{R}^2$, and  
  the inverse image of the origin with the logarithmic map is homeomorphic to the  unknotted circle (i.e. a trivial knot) in $\Log^{-1}(\mathcal{O}) = S^1\times S^1$ isotopic to the diagonal.
\end{proposition}


\begin{proof}

Let $f$ be the polynomial defined by:
$$
f(z,w) = 1 + \tau z +  \tau w + zw,
$$
where $\tau$ is a real number different than $\pm 1$. The curve $\mathcal{C}_f$ with defining polynomial $f$ can be seen as the set of points satisfying the following:
$$
w = - \frac{1+\tau z}{\tau + z} = -\frac{\tau^2(\frac{1}{\tau^2} +\frac{z}{\tau})}{\tau (1+\frac{z}{\tau})} = -\tau \Big(\frac{\frac{1}{\tau^2} +\frac{z}{\tau}}{1+\frac{z}{\tau}} \Big).
$$
Thus,
$$
\frac{w}{\tau} = -  \Big(\frac{\frac{1}{\tau^2} +\frac{z}{\tau}}{1+\frac{z}{\tau}} \Big).
$$
After the change of variable given by $\di \big(\frac{z}{\tau}, \frac{w}{\tau}\big)=(X,Y)$,  we get:
$$
Y = - \frac{a+Y}{1+X},
$$
where $\di a=\frac{1}{\tau^2}$. This is the curve $\mathcal{C}_g$ with defining polynomial $\di g(X,Y)= \frac{1}{\tau^2} +X+Y+XY$.
The amoeba of the curve $\mathcal{C}_f$ is the translation of the amoeba of the curve $\mathcal{C}_g$   by the vector $(\log (\tau), \log(\tau))$.
%
%
%
%
%
%
%

The amoeba and the coamoeba of the curve $\mathcal{C}_g$ with $0<|\tau|<1$ are given as in the figures 4 and   \ref{fig_coamoeba-A}, respectively. Moreover, the point $(-\log (\tau), -\log(\tau))\in \mathbb{R}^2$ is a critical value of the logarithmic map restricted to   the curve $\mathcal{C}_g$ called a pinching point by  Mikhalkin (Remark 10,  \cite{M1-00}). This is equivalent to saying that the origin of $\mathbb{R}^2$ is contained in the amoeba of the curve $\mathcal{C}_f$ as  a pinching point. For more details about this point, we  can see for example \cite{PT-05}.

Thus, the amoeba of $\mathcal{C}_f$ has a critical value at the origin $\mathcal{O}$, which is a pinching point such that its inverse image  by the logarithmic map in $\mathcal{C}_f$ is a circle contained in the real torus $\Log^{-1}(\mathcal{O}) = S^1\times S^1$. In other words, $\mathcal{C}_f\,\cap\,\Log^{-1}(\mathcal{O})$ is a circle. Moreover, the argument map restricted to $\delta = \mathcal{C}_f\,\cap\,\Log^{-1}(\mathcal{O})$ is an homeomorphism into  its image, and its  homology class in $H_1(S^1\times S^1, \mathbb{Z}) \cong \mathbb{Z}\oplus \mathbb{Z}$ is equal to $\left(\begin{array}{c} 1\\ 1\end{array}\right)$ or $\left(\begin{array}{c} -1\\ -1\end{array}\right)$, it depends on the orientation tooken. In fact,  $\Arg (\delta )$ is the union of the two arcs, which is a cycle of homology 
$\left(\begin{array}{c} 1\\ 1\end{array}\right)$ or $\left(\begin{array}{c} -1\\ -1\end{array}\right)$
(see Figure 4).

\end{proof}

\vspace{-0.9cm}

\begin{figure}[H]
 \centering 
  \begin{subfigure}[b]{0.3\textwidth}
    \centering
    \includegraphics[width=0.8\textwidth]{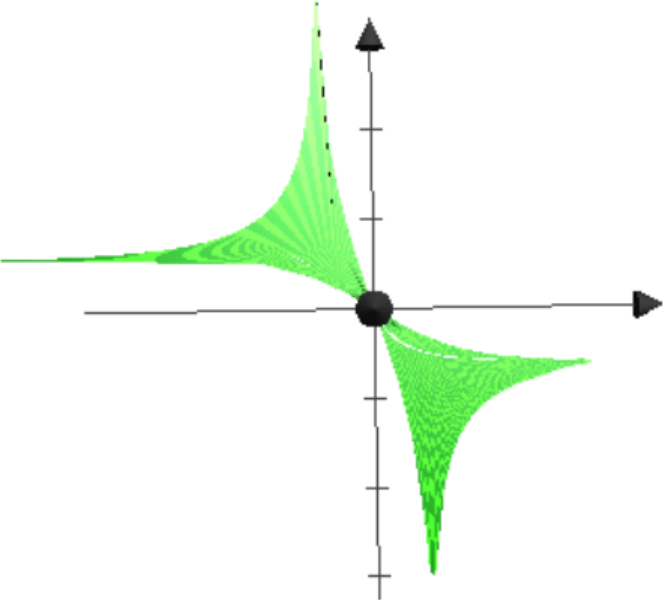}
    \caption{}
    \label{fig_Pinch-Amoeba-A}
  \end{subfigure}
  \begin{subfigure}[b]{0.4\textwidth}
    \centering
    \includegraphics[width=0.5\textwidth]{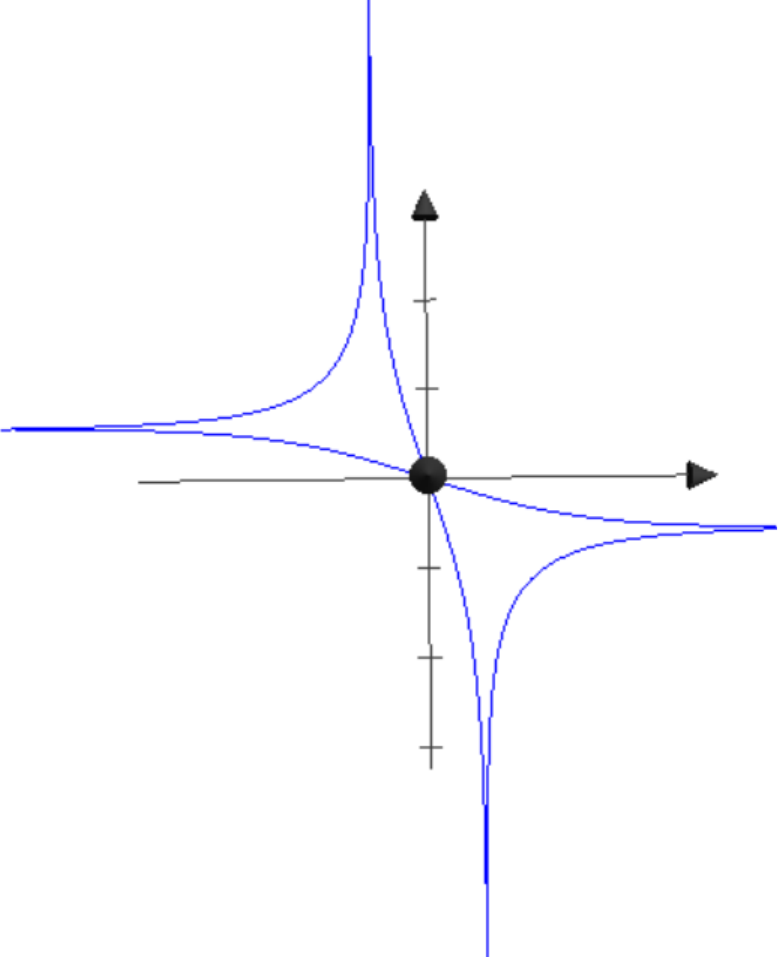}
    \caption{}
    \label{fig_Pinch-Amoeba-AR-1}
  \end{subfigure}
  \caption{(A) Complex amoeba of the  hyperbola defined by $f$ with $0<|\tau|<1$. (B) Image  under the logarithmic map of the real part of the  hyperbola defined by $f$  with $0<|\tau|<1$.}
  \label{fig_Pinch}
\end{figure}
\begin{figure}[H]
 \centering 
  \begin{tikzpicture}[scale=0.21]
        \draw (-6,-6) grid (6,6);

        \draw[line width=0.4mm,green] (-6,0) -- (6,0);
        \draw[line width=0.4mm,black] (0,-6) -- (0,6);  

        \filldraw [gray] (0,0) circle (4pt);
        \draw[line width=0.7mm,red]  (0,6) .. controls (0.5,0.5) .. (6,0);
        \draw[line width=0.7mm,red]  (-6,0) .. controls (-0.5,-0.5) .. (0,-6);   
        \shade[top color=yellow, bottom color=black] (0,0) -- (6,0) .. controls (0.5,0.5) .. (0,6) -- cycle;
        \shade[top color=yellow, bottom color=black] (0,0) -- (-6,0) .. controls (-0.5,-0.5) .. (0,-6) -- cycle;
    \end{tikzpicture}
    \caption{Coamoeba  of the curve $\mathcal{C}_f$ with defining polynomial $f(z,w)=1+\tau z+\tau w + zw$ with $0<|\tau|<1$ in the fundamental domain $[0, 2\pi]\times[0, 2\pi]$ of the real torus $\Log^{-1}(\mathcal{O})=S^1\times S^1$. The intersection of the curve $\mathcal{C}_f$ with $\Log^{-1}(\mathcal{O})$ is a topological circle $\delta$, and the union of the two red  arcs represent the circle image of $\delta$ under the argument map $\Arg(\delta)$.}
    \label{fig_coamoeba-A}
\end{figure}

\vspace{-0.3cm}

We have two possible orientations of the circle  $\Arg(\delta)$ as we can see in the following figures:
\begin{figure}[H]
 \centering 
  \begin{subfigure}[b]{0.4\textwidth}
    \centering
    \includegraphics[width=0.4\textwidth]{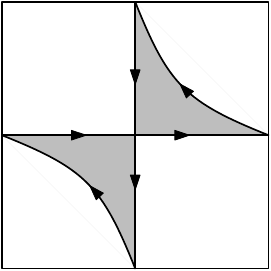}
    \caption{}
    \label{fig_coamoeba.tau.less.1a}
  \end{subfigure}
  \begin{subfigure}[b]{0.4\textwidth}
    \centering
    \includegraphics[width=0.4\textwidth]{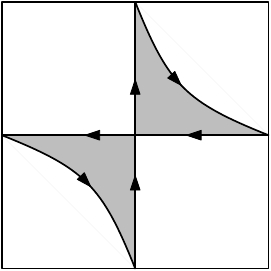}
    \caption{}
    \label{fig_Pinch-Amoeba-AR-1}
  \end{subfigure}
  \caption{(A) Coamoeba of the  hyperbola defined by $f$ with $0<|\tau|<1$, and  the class of the 1-cycle $\Arg(\delta)$ is $[\Arg(\delta)]= \left(\begin{array}{c} -1\\ 1\end{array}\right)$ in $H_1(S^1\times S^1,\mathbb{Z})$.
  (B) Coamoeba of the  hyperbola defined by $f$ with $0<|\tau|<1$, and  the class of the 1-cycle $\Arg(\delta)$ is $[\Arg(\delta)]= \left(\begin{array}{c} 1\\ -1\end{array}\right)$ in $H_1(S^1\times S^1,\mathbb{Z})$.}
  \label{fig_Pinch}
\end{figure}

 \vspace{0.0cm}
In the case where $|\tau| >1$, the amoeba and the coamoeba of the curve defined by $f$ are as shown in figures 7,  8 and 9:

%
%
%
%
%
%
%

\begin{figure}[H]
 \centering 
  \begin{subfigure}[b]{0.47\textwidth}
    \centering
    \includegraphics[width=1.7\textwidth]{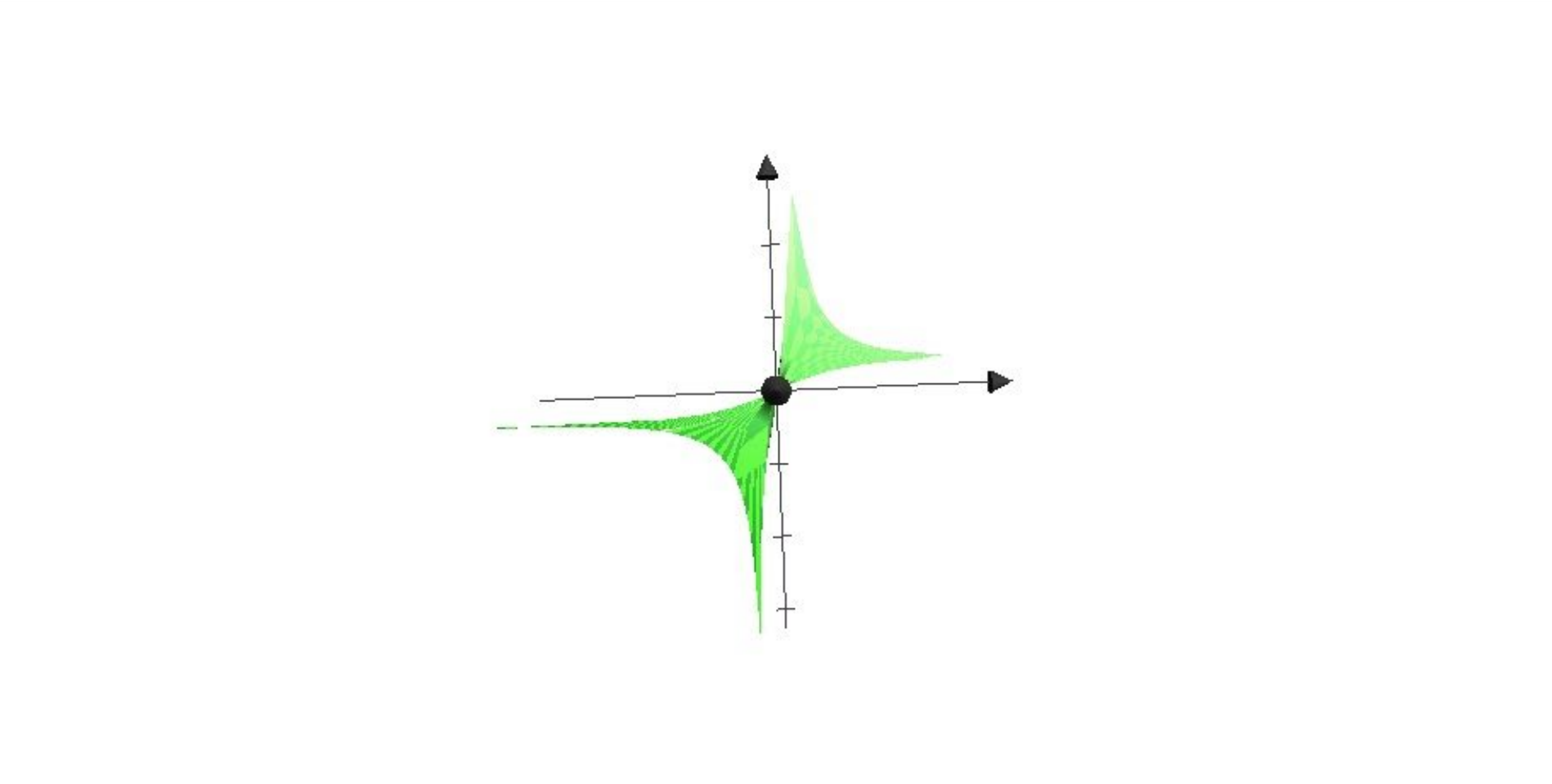}
    \caption{}
    \label{fig_Pinch-Amoeba-A}
  \end{subfigure}
  \begin{subfigure}[b]{0.5\textwidth}
    \centering
    \includegraphics[width=1.5\textwidth]{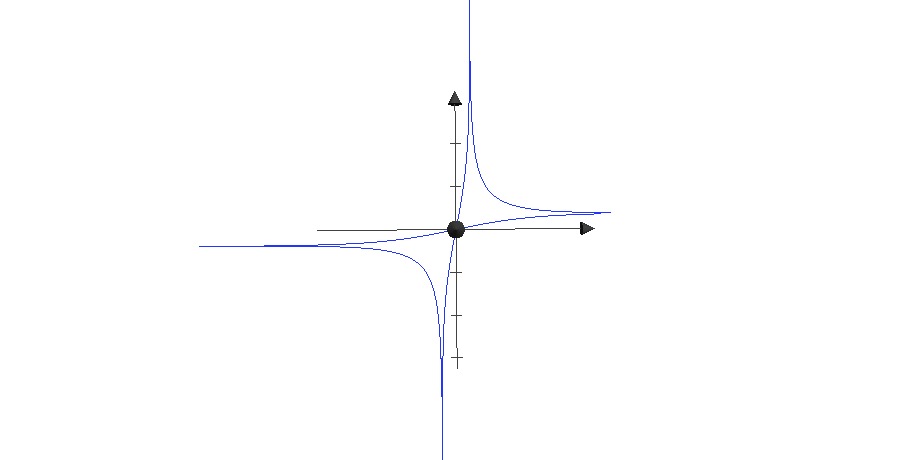}
    \caption{}
    \label{fig_Pinch-Amoeba-AR-1}
  \end{subfigure}
  \caption{(A) Complex amoeba of the  hyperbola defined by $f$ with $|\tau|>1$. (B) Image  under the logarithmic map of the real part of the  hyperbola defined by $f$ with $|\tau|>1$.}
  \label{fig_Pinch2}
\end{figure}
\vspace{-0.0cm}

\begin{figure}[H]
\begin{center}
 \centering 
  \begin{tikzpicture}[scale=0.21]
        \draw (-6,-6) grid (6,6);
 \draw[line width=0.4mm,green] (-6,0) -- (6,0);
        \draw[line width=0.4mm,black] (0,-6) -- (0,6);

        \filldraw [gray] (0,0) circle (4pt);
         \draw[line width=0.7mm,red] (6,0) .. controls (0.5,-0.5) .. (0,-6);
         \draw[line width=0.7mm,red] (-6,0) .. controls (-0.5,0.5) .. (0,6);
        \shade[top color=yellow, bottom color=black] (0,0) -- (-6,0) .. controls (-0.5,0.5) .. (0,6) -- cycle;
        \shade[top color=yellow, bottom color=black] (0,0) -- (0,-6) .. controls (0.5,-0.5) .. (6,0) -- cycle;
    \end{tikzpicture}
    \caption{Coamoeba  of the curve $\mathcal{C}_f$ with defining polynomial $f(z,w)=1+\tau z+\tau w + zw$ with $|\tau|>1$ in the fundamental domain $[0, 2\pi]\times[0, 2\pi]$ of the real torus $\Log^{-1}(\mathcal{O})=(S^1)^2$. The intersection of the curve $\mathcal{C}_f$ with $\Log^{-1}(\mathcal{O})$ is a topological circle $\delta$, and the union of the two red arcs represent the circle image of $\delta$ under the argument map $\Arg(\delta)$.} 
    \label{c}
\end{center}
\end{figure} 
Similarly to the first case, we have two possible orientations of the circle  $\Arg(\delta)$ as we can see in the following figures:
\begin{figure}[H]
 \centering 
  \begin{subfigure}[b]{0.4\textwidth}
    \centering
    \includegraphics[width=0.4\textwidth]{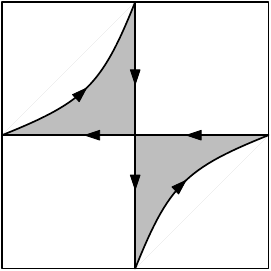}
    \caption{}
    \label{fig_coamoeba.tau.less.1a}
  \end{subfigure}
  \begin{subfigure}[b]{0.4\textwidth}
    \centering
    \includegraphics[width=0.4\textwidth]{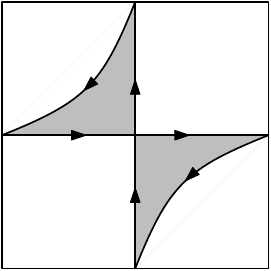}
    \caption{}
    \label{fig_Pinch-Amoeba-AR-1}
  \end{subfigure}
  \caption{(A) Coamoeba of the  hyperbola defined by $f$ with $|\tau |>1$, and the class of the 1-cycle $\Arg(\delta)$ is $[\Arg(\delta)]= \left(\begin{array}{c} 1\\ 1\end{array}\right)$ in $H_1(S^1\times S^1,\mathbb{Z})$.
  (B) Coamoeba of the  hyperbola defined by $f$ with $|\tau| >1$, and the class of the 1-cycle $\Arg(\delta)$ is  $[\Arg(\delta)]= \left(\begin{array}{c} -1\\ -1\end{array}\right)$ in $H_1(S^1\times S^1,\mathbb{Z})$.}
  \label{fig_Pinch}
\end{figure}

%
%
%
%
%
%
%
%
 

{\it Proof of Theorem 4.5.} First, consider the case when $|\tau|>1$. By hypothesis the polynomial $g$ is equal to $f_L$, and  by Lemma 2.1,  the amoeba (resp.  coamoeba) of the curve $\mathcal{C}_g$ is the image under the linear transformation $^{t}L^{-1}$ of the amoeba (resp. coamoeba) of the hyperbola with defining polynomial $f$ as in Example 3.6. This induces an homomorphism of the homology group $H_1(S^1\times S^1, \mathbb{Z})$ to itself given by:
$$
^{t}\textrm{Adj}(L) : H_1(S^1\times S^1, \mathbb{Z}) \longrightarrow H_1(S^1\times S^1, \mathbb{Z}),
$$
\vspace{-0.1cm}
where $^{t}\textrm{Adj}(L)$ is the transpose of the adjoint matrix of $L$ which is equal to $\di 
^{t}\textrm{Adj}(L) = \left(\begin{array}{cc} d&-c\\ -b&a\end{array}\right).
$
Let $\delta = \mathcal{C}_f\cap \Log^{-1}(\mathcal{O})$ be the circle representing the set of critical points of the logarithmic map over the origin.  By Corollary 3.2, $\delta$ is contained in the set of critical points of the argument map. 
The class of the image of the 1-cycle $\delta$ under the argument map  $[\Arg (\delta)]= \left(\begin{array}{c} 1\\ 1\end{array}\right)$ or $[\Arg (\delta)] = \left(\begin{array}{c} -1\\ -1\end{array}\right)$ in the coamoeba of the hyperbola  is precisely the one with homology $\left(\begin{array}{c} d-c\\ a-b\end{array}\right)$ or $\left(\begin{array}{c} -d+c\\ -a+b\end{array}\right)$. The integer length of $\left(\begin{array}{c} d-c\\ a-b\end{array}\right)$ is equal the greatest common divisor of $d-c$ and $a-b$. 

\noindent The proof in the case where $0<|\tau|<1$ is done in the same way using the hyperbola in Proposition 4.6. In this case  the   homology class of $[\Arg (\delta )]$ in $H_1(S^1\times S^1, \mathbb{Z}) \cong \mathbb{Z}\oplus \mathbb{Z}$ is equal to $\left(\begin{array}{c} -1\\ 1\end{array}\right)$ or $\left(\begin{array}{c} 1\\ -1\end{array}\right)$.

\begin{figure}[H]
 \centering 
  \begin{subfigure}[b]{0.49\textwidth}
    \centering
    \includegraphics[width=0.5\textwidth]{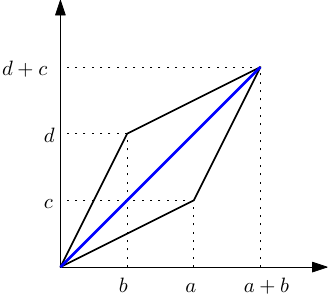}
    \caption{}
    \label{fig_coamoeba.tau.less.1a}
  \end{subfigure}
  \begin{subfigure}[b]{0.5\textwidth}
    \centering
    \includegraphics[width=0.5\textwidth]{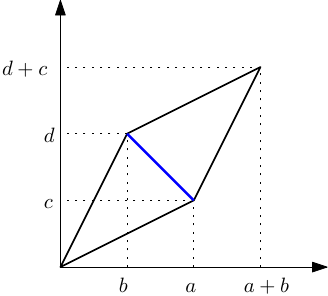}
    \caption{}
    \label{fig_Pinch-Amoeba-AR-1}
  \end{subfigure}
  \caption{(A) The diagonal of the Newton polygon of the defining polynomial of $f$ when $0<|\tau| <1$. (B) The diagonal of the Newton polygon of the defining polynomial of $f$ when $|\tau| >1$.}
  \label{fig_Pinch}
\end{figure}

\vspace{0.3cm}

\begin{corollary}
Any toric link can be realized as the one-dimensional connected components of the critical values of the argument map restricted to a complex  algebraic plane curve.
\end{corollary}


\begin{example} Here are two examples of knots in the coamoebas of a complex algebraic plane curves. 
Using Lemma 2.1, in Figure 11. (A), respectively (B), we draw in red the image of the cycle $\Arg(\delta)$ by the matrix $^{t}L_f^{-1}$, respectively $^{t}L_g^{-1}$, where $L_f$  and $L_g$ are the matrices given by:
$$
L_f = \left(\begin{array}{cc} 2&2\\ 2&3\end{array}\right), \,\textrm{and}\quad L_g = \left(\begin{array}{cc} 3&1\\ 3&2\end{array}\right).
$$

\vspace{0.1cm}

 \begin{figure}[H]
   \begin{subfigure}[b]{0.49\textwidth}
    \centering
    \begin{tikzpicture}[scale=0.25]
        \draw (-6,-6) grid (6,6);

        \draw[line width=0.4mm,green] (-6,-4) -- (-3,-6);
        \draw[line width=0.4mm,green] (-6,0) -- (3,-6);
        \draw[line width=0.4mm,green] (-6,4) -- (6,-4);
        \draw[line width=0.4mm,green] (-3,6) -- (6,0);
        \draw[line width=0.4mm,green] (3,6) -- (6,4);

        \draw[line width=0.4mm,black] (-6,-3) -- (-3,-6);
        \draw[line width=0.4mm,black] (-6,3) -- (3,-6);
        \draw[line width=0.4mm,black] (-3,6) -- (6,-3);
        \draw[line width=0.4mm,black] (3,6) -- (6,3);

        \draw[line width=0.7mm,red]  (-6,0) .. controls (-4.5,-0.5) .. (-3,0);
        \draw[line width=0.7mm,red]  (-3,0) .. controls (-1.5,0.5) .. (0,0);

        \draw[line width=0.7mm,red]  (0,0) .. controls (1.5,-0.5) .. (3,0);
        \draw[line width=0.7mm,red]  (3,0) .. controls (4.5,0.5) .. (6,0);
 
        \shade[top color=yellow, bottom color=black] (-6,-4) -- (-3,-6) -- (-6,-3) -- cycle;
        \shade[top color=yellow, bottom color=black] (3,-6) -- (-3,0) .. controls (-4.5,-0.5) .. (-6,0) -- cycle;
  
        \shade[top color=yellow, bottom color=black] (-6,3) -- (-6,4) -- (0,0) .. controls (-1.5,0.5) .. (-3,0) -- cycle; 
   
        \shade[top color=yellow, bottom color=black] (6,-4) -- (6,-3) -- (3,0) .. controls (1.5,-0.5) .. (0,0) -- cycle; 
    
        \shade[top color=yellow, bottom color=black] (3,0) .. controls (4.5,0.5) .. (6,0)  -- (-3,6 )-- cycle; 
        
        \shade[top color=yellow, bottom color=black] (6,3) -- (6,4) -- (3,6) -- cycle;


\end{tikzpicture}
    \caption{$f(z,w)=1+\tau z^2w^2+\tau z^2w^3 + z^4w^5$.}
    \label{fig_2223}
  \end{subfigure}
  \begin{subfigure}[b]{0.49\textwidth}
  \centering
  \begin{tikzpicture}[scale=0.25]
    \draw (-6,-6) grid (6,6);


    \draw[line width=0.4mm,black] (-6,-4) -- (-4,-6);  
    \draw[line width=0.4mm,black] (6,-4) -- (-4,6);  
    \draw[line width=0.4mm,black] (-6,0) -- (0,-6);  
    \draw[line width=0.4mm,black] (0,6) -- (6,0);   
    \draw[line width=0.4mm,black] (-6,4) -- (4,-6);  
    \draw[line width=0.4mm,black] (4,6) -- (6,4);  

    \draw[line width=0.4mm,green] (-6,-3) -- (0,-6); 
    \draw[line width=0.4mm,green] (0,6) -- (6,3); 
    \draw[line width=0.4mm,green] (-6,3) -- (6,-3);

    \draw[line width=0.7mm,red]  (-2,4) .. controls (-0.2,1.7) .. (0,0);

    \draw[line width=0.7mm,red]  (2,-4) .. controls (0.2,-1.7) .. (0,0);

    \draw[line width=0.7mm,red]  (2,-4) .. controls (3.2,-5.4) .. (3.4,-6);

    \draw[line width=0.7mm,red]  (4,4) .. controls (3.8,5.1) .. (3.4,6);

    \draw[line width=0.7mm,red]  (4,4) .. controls (4.2,2.5) .. (6,0);

    \draw[line width=0.7mm,red]  (-4,-4) .. controls (-3.8,-5.1) .. (-3.4,-6);

    \draw[line width=0.7mm,red]  (-2,4) .. controls (-3.2,5.4) .. (-3.4,6);

    \draw[line width=0.7mm,red]  (-4,-4) .. controls (-4.2,-2.5) .. (-6,0);
 

    \shade[top color=yellow, bottom color=black] (0,-6) -- (-4,-4) .. controls (-4.2,-2.5) .. (-6,0);

    \shade[top color=yellow, bottom color=black] (-4,2) -- (2,-4) .. controls (0.2,-1.7) .. (0,0);

    \shade[top color=yellow, bottom color=black] (4,-2) -- (-2,4) .. controls (-0.2,1.7) .. (0,0);

    \shade[top color=yellow, bottom color=black] (0,6) -- (4,4) .. controls (4.2,2.5) .. (6,0);

    \shade[top color=yellow, bottom color=black] (6,3) -- (6,4) -- (4,6) -- (3.4,6) .. controls (3.8,5.1) .. (4,4);

    \shade[top color=yellow, bottom color=black] (-6,-3) -- (-6,-4) -- (-4,-6) -- (-3.4,-6) .. controls (-3.8,-5.1) .. (-4,-4);

    \shade[top color=yellow, bottom color=black] (-6,4) -- (-4,2) -- (-6,3);

    \shade[top color=yellow, bottom color=black] (6,-4) -- (4,-2) -- (6,-3);

    \shade[top color=yellow, bottom color=black] (2,-4) .. controls (3.2,-5.4) .. (3.4,-6) -- (4,-6);

    \shade[top color=yellow, bottom color=black] (-2,4) .. controls (-3.2,5.4) .. (-3.4,6) -- (-4,6);
 
\end{tikzpicture}
  \caption{$g(z,w)=1+\tau z^3w^3+\tau zw^2 + z^4w^5$.}
    \label{fig_3312}
  \end{subfigure}
  \caption{Coamoebas  in a fundamental domain of the 2-dimensional real torus $S^1\times S^1$ of the curves with defining polynomials $f$ and $g$ respectively with $0<|\tau|<1$.}
  \label{fig_coamoeba1}
\end{figure}

\end{example}

\vspace{0.2cm}
\newpage

\begin{corollary}
Let $f_{(p,q)}$ be the polynomial defined as follows:
$$
f_{(p,q)}(z,w) = 1 +\tau z^p + \tau w^q + z^pw^q,
$$
    where $\tau$ is a real number such that $\tau\ne \pm 1$, $p$ and $q$ are positive integer numbers with $(p,q) = c$. Then, the amoeba of the curve $\mathcal{C}_{f_{(p,q)}}$ with defining polynomial $f_{(p,q)}$ contains the origin  $\mathcal{O}\in \mathbb{R}^2$, and  the inverse image of the origin with the logarithmic map is as follows:
\begin{itemize}
\item[(i)]\, If $(p,q) = 1$ (i.e. $p$ and $q$ are coprime), then $\Log^{-1}(\mathcal{O})\cap V_{f_{(p,q)}}$ is a $(p,q)$-torus knot.
\item[(ii)]\, if $(p,q) = c\ne 1$, then $\Log^{-1}(\mathcal{O})\,\cap\, V_{f_{(p,q)}}$ is a $(p,q)$-torus link with $c$ connected components.

\end{itemize}
\end{corollary}

\begin{proof} 
The coamoeba when $p=q=1$ is given in Figure 4. Using Lemma \ref{lem1}, and the fact that $f_{(p,q)} = f_D$ where $D$ is the diagonal matrix with coefficients $p$ and $q$  the result follows immediately.
\end{proof}

\section{Realization of some known knots in coamoebas of algebraic curves}

 In this section, we gave an explicit representation of some known knots and links as critical values of the argument map restricted to  a complex algebraic plane curves.   This is achieved via an explicit construction of complex algebraic plane curves  by writing down their defining polynomials.

\begin{example} ${}$

\begin{itemize}
\item[(1)]\, In this example we construct the {\em Hopf Link},  the link  consisting of a linked two circles $S^1 \cup S^1$ with linking number $\pm 1$. Let $V_f$ be the curve with defining polynomial $fz,w)= 1 + \tau z^2 + \tau w^2 + zw$. The set of critical values of the argument map restricted to the curve  $V_f$ is the image of the cycle $\Arg(\delta)$ by the matrix $^{t}L^{-1}$ where $L$ is the diagonal matrix  $2\times {\textrm Id}_{2}$.

\begin{figure}[H]
 \begin{subfigure}[b]{0.49\textwidth}
    \centering
    \begin{tikzpicture}[scale=0.25]
        \draw (-6,-6) grid (6,6);

        \draw[line width=0.4mm,green] (-3,-6) -- (-3,6);
        \draw[line width=0.4mm,green] (3,-6) -- (3,6);

        \draw[line width=0.4mm,black] (-6,-3) -- (6,-3);
        \draw[line width=0.4mm,black] (-6,3) -- (6,3);

        \draw[line width=0.7mm,red]  (-3,-6) .. controls (-3.5,-3.5) .. (-6,-3);
        \draw[line width=0.7mm,red] (-3,0) .. controls (-2.5,-2.5) .. (0,-3);

        \draw[line width=0.7mm,red] (-3,0) .. controls (-3.5,2.5) .. (-6,3);
        \draw[line width=0.7mm,red] (0,3) .. controls (-2.5,3.5) .. (-3,6);
        \draw[line width=0.7mm,red] (3,-6) .. controls (2.5,-3.5) .. (0,-3);
        \draw[line width=0.7mm,red] (6,-3) .. controls (3.5,-2.5) .. (3,0);
        \draw[line width=0.7mm,red] (3,0) .. controls (2.5,2.5) .. (0,3);
        \draw[line width=0.7mm,red] (6,3) .. controls (3.5,3.5) .. (3,6);
        \shade[top color=yellow, bottom color=black] (-3,-3) -- (-3,-6) .. controls (-3.5,-3.5) .. (-6,-3) -- cycle;
        \shade[top color=yellow, bottom color=black] (-3,-3) -- (-3,0) .. controls (-2.5,-2.5) .. (0,-3) -- cycle;
        \shade[top color=yellow, bottom color=black] (-3,3) -- (-3,0) .. controls (-3.5,2.5) .. (-6,3) -- cycle;
        \shade[top color=yellow, bottom color=black] (-3,3) -- (0,3) .. controls (-2.5,3.5) .. (-3,6) -- cycle;
 
        \shade[top color=yellow, bottom color=black] (3,-3) -- (3,-6) .. controls (2.5,-3.5) .. (0,-3) -- cycle;
        \shade[top color=yellow, bottom color=black] (3,-3) -- (6,-3) .. controls (3.5,-2.5) .. (3,0) -- cycle;

        \shade[top color=yellow, bottom color=black] (3,3) -- (3,0) .. controls (2.5,2.5) .. (0,3) -- cycle;
        \shade[top color=yellow, bottom color=black] (3,3) -- (6,3) .. controls (3.5,3.5) .. (3,6) -- cycle;   


    \end{tikzpicture}
    \caption{$0<|\tau|<1$.}
    \label{fig_2002a}
  \end{subfigure}
  \begin{subfigure}[b]{0.49\textwidth}
    \centering
    \begin{tikzpicture}[scale=0.25]
        \draw (-6,-6) grid (6,6);

        \draw[line width=0.4mm,green] (-3,-6) -- (-3,6);
        \draw[line width=0.4mm,green] (3,-6) -- (3,6);

        \draw[line width=0.4mm,black] (-6,-3) -- (6,-3);
        \draw[line width=0.4mm,black] (-6,3) -- (6,3);

        \draw[line width=0.7mm,red]  (-3,0) .. controls (-3.5,-2.5) .. (-6,-3);
        \draw[line width=0.7mm,red] (-3,-6) .. controls (-2.5,-3.5) .. (0,-3);

        \draw[line width=0.7mm,red] (-3,0) .. controls (-2.5,2.5) .. (0,3);
        \draw[line width=0.7mm,red] (-6,3) .. controls (-3.5,3.5) .. (-3,6);

        \draw[line width=0.7mm,red]  (3,0) .. controls (2.5,-2.5) .. (0,-3);
        \draw[line width=0.7mm,red] (3,-6) .. controls (3.5,-3.5) .. (6,-3);

        \draw[line width=0.7mm,red] (3,0) .. controls (3.5,2.5) .. (6,3);
        \draw[line width=0.7mm,red] (0,3) .. controls (2.5,3.5) .. (3,6);

        \shade[top color=yellow, bottom color=black] (-3,-3) -- (-3,0) .. controls (-3.5,-2.5) .. (-6,-3) -- cycle;
        \shade[top color=yellow, bottom color=black] (-3,-3) -- (-3,-6) .. controls (-2.5,-3.5) .. (0,-3) -- cycle;

        \shade[top color=yellow, bottom color=black] (3,-3) -- (3,0) .. controls (2.5,-2.5) .. (0,-3) -- cycle;
        \shade[top color=yellow, bottom color=black] (3,-3) -- (3,-6) .. controls (3.5,-3.5) .. (6,-3) -- cycle;

        \shade[top color=yellow, bottom color=black] (-3,3) -- (-3,6) .. controls (-3.5,3.5) .. (-6,3) -- cycle;
        \shade[top color=yellow, bottom color=black] (-3,3) -- (-3,0) .. controls (-2.5,2.5) .. (0,3) -- cycle;

        \shade[top color=yellow, bottom color=black] (3,3) -- (3,6) .. controls (2.5,3.5) .. (0,3) -- cycle;
        \shade[top color=yellow, bottom color=black] (3,3) -- (3,0) .. controls (3.5,2.5) .. (6,3) -- cycle;



    \end{tikzpicture}
    \caption{$|\tau|>1$.}
    \label{fig_2002b}
  \end{subfigure}

 \caption{Coamoeba of the curve $V_f$  with defining polynomial $f(z,w)=1+\tau z^2+\tau w^2+z^2w^2$ with (A) $0<|\tau|<1$ and (B) when  $|\tau|>1$. The red curve is the set of critical values of the argument map restricted to $V_f$ which  is the {\em Hopf Link}.}
 \label{fig_2002}
\end{figure}

\begin{figure}[H]
    \centering
\begin{tikzpicture}[scale=0.80]
   \draw [thick,domain=0:217] plot ({cos(\x)},{sin(\x)});
   \draw [thick,domain=227:360] plot ({cos(\x)}, {sin(\x)});
      \draw [thick,domain=0:36] plot ({-1.5+cos(\x)}, {sin(\x)});
   \draw [thick,domain=47:360] plot ({-1.5+cos(\x)}, {sin(\x)});
 %
\end{tikzpicture}
\caption{Hopf link}
\end{figure}

%
%
%
%


\item[(2)]\, Here is the example of the so-called {\em Borromean ring}.

\begin{figure}[H]
 \begin{subfigure}[b]{0.49\textwidth}
    \centering
    \begin{tikzpicture}[scale=0.25]
        \draw (-6,-6) grid (6,6);

        \draw[line width=0.4mm,green] (-4,-6) -- (-4,6);
        \draw[line width=0.4mm,green] (0,-6) -- (0,6);
        \draw[line width=0.4mm,green] (4,-6) -- (4,6);
        \draw[line width=0.4mm,black] (-6,-4) -- (6,-4);
        \draw[line width=0.4mm,black] (-6,0) -- (6,0);
        \draw[line width=0.4mm,black] (-6,4) -- (6,4);

        \draw[line width=0.7mm,red]  (-4,-6) .. controls (-4.5,-4.5) .. (-6,-4);
        \draw[line width=0.7mm,red] (-2,-4) .. controls (-3.5,-3.5) .. (-4,-2);
        \draw[line width=0.7mm,red]  (-2,-4) .. controls (-0.5,-4.5) .. (0,-6);
        \draw[line width=0.7mm,red] (2,-4) .. controls (0.5,-3.5) .. (0,-2);
        \draw[line width=0.7mm,red] (2,-4) .. controls (3.5,-4.5) .. (4,-6);
        \draw[line width=0.7mm,red] (6,-4) .. controls (4.5,-3.5) .. (4,-2);
        \draw[line width=0.7mm,red] (-4,-2) .. controls (-4.5,-0.5) .. (-6,0);
        \draw[line width=0.7mm,red] (-2,0) .. controls (-3.5,0.5) .. (-4,2);
        \draw[line width=0.7mm,red] (0,-2) .. controls (-0.5,-0.5) .. (-2,0);
        \draw[line width=0.7mm,red] (2,0) .. controls (0.5,0.5) .. (0,2);
        \draw[line width=0.7mm,red] (4,-2) .. controls (3.5,-0.5) .. (2,0);
        \draw[line width=0.7mm,red] (6,0) .. controls (4.5,0.5) .. (4,2);
        \draw[line width=0.7mm,red] (-4,2) .. controls (-4.5,3.5) .. (-6,4);
        \draw[line width=0.7mm,red] (-2,4) .. controls (-3.5,4.5) .. (-4,6);
        \draw[line width=0.7mm,red] (2,4) .. controls (0.5,4.5) .. (0,6);
        \draw[line width=0.7mm,red] (-2,4) .. controls (-0.5,3.5) .. (0,2);
        \draw[line width=0.7mm,red] (2,4) .. controls (3.5,3.5) .. (4,2);
        \draw[line width=0.7mm,red] (6,4) .. controls (4.5,4.5) .. (4,6);

        \shade[top color=yellow, bottom color=black] (-4,-4) -- (-4,-6) .. controls (-4.5,-4.5) .. (-6,-4) -- cycle;
        \shade[top color=yellow, bottom color=black] (-4,-4) -- (-2,-4) .. controls (-3.5,-3.5) .. (-4,-2) -- cycle;

        \shade[top color=yellow, bottom color=black] (-4,0) -- (-4,-2) .. controls (-4.5,-0.5) .. (-6,0) -- cycle;
        \shade[top color=yellow, bottom color=black] (-4,0) -- (-2,0) .. controls (-3.5,0.5) .. (-4,2) -- cycle;

        \shade[top color=yellow, bottom color=black] (-4,4) -- (-4,2) .. controls (-4.5,3.5) .. (-6,4) -- cycle;
        \shade[top color=yellow, bottom color=black] (-4,4) -- (-2,4) .. controls (-3.5,4.5) .. (-4,6) -- cycle;

        \shade[top color=yellow, bottom color=black] (0,-4) -- (0,-6) .. controls (-0.5,-4.5) .. (-2,-4) -- cycle;
        \shade[top color=yellow, bottom color=black] (0,-4) -- (2,-4) .. controls (0.5,-3.5) .. (0,-2) -- cycle;

        \shade[top color=yellow, bottom color=black] (0,0) -- (0,-2) .. controls (-0.5,-0.5) .. (-2,0) -- cycle;
        \shade[top color=yellow, bottom color=black] (0,0) -- (2,0) .. controls (0.5,0.5) .. (0,2) -- cycle;

        \shade[top color=yellow, bottom color=black] (0,4) -- (0,2) .. controls (-0.5,3.5) .. (-2,4) -- cycle;
        \shade[top color=yellow, bottom color=black] (0,4) -- (2,4) .. controls (0.5,4.5) .. (0,6) -- cycle;
        \shade[top color=yellow, bottom color=black] (4,-4) -- (4,-6) .. controls (3.5,-4.5) .. (2,-4) -- cycle;
        \shade[top color=yellow, bottom color=black] (4,-4) -- (6,-4) .. controls (4.5,-3.5) .. (4,-2) -- cycle;

        \shade[top color=yellow, bottom color=black] (4,0) -- (4,-2) .. controls (3.5,-0.5) .. (2,0) -- cycle;
        \shade[top color=yellow, bottom color=black] (4,0) -- (6,0) .. controls (4.5,0.5) .. (4,2) -- cycle;

        \shade[top color=yellow, bottom color=black] (4,4) -- (4,2) .. controls (3.5,3.5) .. (2,4) -- cycle;
        \shade[top color=yellow, bottom color=black] (4,4) -- (6,4) .. controls (4.5,4.5) .. (4,6) -- cycle;
    \end{tikzpicture}
    \caption{$0<|\tau|<1$.}
    \label{fig_3003a}
  \end{subfigure}
  \begin{subfigure}[b]{0.49\textwidth}
    \centering
    \begin{tikzpicture}[scale=0.25]
        \draw (-6,-6) grid (6,6);

        \draw[line width=0.4mm,green] (-4,-6) -- (-4,6);
        \draw[line width=0.4mm,green] (0,-6) -- (0,6);
        \draw[line width=0.4mm,green] (4,-6) -- (4,6);
        \draw[line width=0.4mm,black] (-6,-4) -- (6,-4);
        \draw[line width=0.4mm,black] (-6,0) -- (6,0);
        \draw[line width=0.4mm,black] (-6,4) -- (6,4);

        \draw[line width=0.7mm,red]  (-4,-2) .. controls (-4.5,-3.5) .. (-6,-4);
        \draw[line width=0.7mm,red] (-2,-4) .. controls (-3.5,-4.5) .. (-4,-6);
        \draw[line width=0.7mm,red]  (-4,2) .. controls (-4.5,0.5) .. (-6,0);
        \draw[line width=0.7mm,red] (-2,0) .. controls (-3.5,-0.5) .. (-4,-2);
        \draw[line width=0.7mm,red]  (-4,6) .. controls (-4.5,4.5) .. (-6,4);
        \draw[line width=0.7mm,red] (-2,4) .. controls (-3.5,3.5) .. (-4,2);
        \draw[line width=0.7mm,red]  (0,-2) .. controls (-0.5,-3.5) .. (-2,-4);
        \draw[line width=0.7mm,red] (2,-4) .. controls (0.5,-4.5) .. (0,-6);
        \draw[line width=0.7mm,red]  (0,2) .. controls (-0.5,0.5) .. (-2,0);
        \draw[line width=0.7mm,red] (2,0) .. controls (0.5,-0.5) .. (0,-2);
        \draw[line width=0.7mm,red]  (0,6) .. controls (-0.5,4.5) .. (-2,4);
        \draw[line width=0.7mm,red] (2,4) .. controls (0.5,3.5) .. (0,2);
        \draw[line width=0.7mm,red]  (4,-2) .. controls (3.5,-3.5) .. (2,-4);
        \draw[line width=0.7mm,red] (6,-4) .. controls (4.5,-4.5) .. (4,-6);
        \draw[line width=0.7mm,red]  (4,2) .. controls (3.5,0.5) .. (2,0);
        \draw[line width=0.7mm,red] (6,0) .. controls (4.5,-0.5) .. (4,-2);
        \draw[line width=0.7mm,red]  (4,6) .. controls (3.5,4.5) .. (2,4);
        \draw[line width=0.7mm,red] (6,4) .. controls (4.5,3.5) .. (4,2);
        \shade[top color=yellow, bottom color=black] (4,-4) -- (6,-4) .. controls (4.5,-4.5) .. (4,-6)-- cycle;
        \shade[top color=yellow, bottom color=black] (4,-4) -- (4,-2) .. controls (3.5,-3.5) .. (2,-4)-- cycle;

        \shade[top color=yellow, bottom color=black] (4,0) -- (4,2) .. controls (3.5,0.5) .. (2,0) -- cycle;
        \shade[top color=yellow, bottom color=black] (4,0) -- (6,0) .. controls (4.5,-0.5) .. (4,-2)-- cycle;

        \shade[top color=yellow, bottom color=black] (4,4) -- (4,6) .. controls (3.5,4.5) .. (2,4) -- cycle;
        \shade[top color=yellow, bottom color=black] (4,4) -- (6,4) .. controls (4.5,3.5) .. (4,2) -- cycle;

        \shade[top color=yellow, bottom color=black] (0,-4) -- (2,-4) .. controls (0.5,-4.5) .. (0,-6)-- cycle;
        \shade[top color=yellow, bottom color=black] (0,-4) -- (0,-2) .. controls (-0.5,-3.5) .. (-2,-4)-- cycle;

        \shade[top color=yellow, bottom color=black] (0,0) -- (0,2) .. controls (-0.5,0.5) .. (-2,0) -- cycle;
        \shade[top color=yellow, bottom color=black] (0,0) -- (2,0) .. controls (0.5,-0.5) .. (0,-2)-- cycle;

        \shade[top color=yellow, bottom color=black] (0,4) -- (0,6) .. controls (-0.5,4.5) .. (-2,4) -- cycle;
        \shade[top color=yellow, bottom color=black] (0,4) -- (2,4) .. controls (0.5,3.5) .. (0,2) -- cycle;

        \shade[top color=yellow, bottom color=black] (-4,-4) -- (-2,-4) .. controls (-3.5,-4.5) .. (-4,-6)-- cycle;
        \shade[top color=yellow, bottom color=black] (-4,-4) -- (-4,-2) .. controls (-4.5,-3.5) .. (-6,-4)-- cycle;

        \shade[top color=yellow, bottom color=black] (-4,0) -- (-4,2) .. controls (-4.5,0.5) .. (-6,0) -- cycle;
        \shade[top color=yellow, bottom color=black] (-4,0) -- (-2,0) .. controls (-3.5,-0.5) .. (-4,-2)-- cycle;

        \shade[top color=yellow, bottom color=black] (-4,4) -- (-4,6) .. controls (-4.5,4.5) .. (-6,4) -- cycle;
        \shade[top color=yellow, bottom color=black] (-4,4) -- (-2,4) .. controls (-3.5,3.5) .. (-4,2) -- cycle;
    \end{tikzpicture}
    \caption{$|\tau|>1$.}
    \label{fig_3003b}
  \end{subfigure}
%
  %
  %
 \caption{Coamoeba of the curve $V_f$ with defining polynomial $f(z,w)=1+\tau z^3+\tau w^3+z^3w^3$ with (A) $0<|\tau|<1$ and (B) $|\tau|>1$.  The union of the one-dimensional connected components of the critical values of the argument map restricted to the curve $V_f$ is precisely the so-called {\em Borromean rings}.}
 \label{fig_3003}
\end{figure}

\begin{figure}[H]
\centering
\def\firstcircle{ (0.0, 0.0) circle (0.5)}
\def\secondcircle{(0.8, 0.0) circle (0.5)}
\def\thirdcircle{ (0.5,-0.5) circle (0.5)}
\def\rectangle{ (-1.5,-3.0) rectangle (3.5,1.0) }
\colorlet{circle edge}{black}
\colorlet{circle area}{white}

\tikzset{filled/.style={fill=circle area, draw=circle edge, thick},
    outline/.style={draw=circle edge, thick}}

\begin{tikzpicture}[scale=0.90]
 \draw [thick,domain=0:217] plot ({cos(\x)},{sin(\x)});
   
   \draw [thick,domain=227:266] plot ({cos(\x)}, {sin(\x)});
   \draw [thick,domain=274:360] plot ({cos(\x)}, {sin(\x)});
   
      \draw [thick,domain=0:36] plot ({-1.5+cos(\x)}, {sin(\x)});
      \draw [thick,domain=356:360] plot ({-1.5+cos(\x)}, {sin(\x)});
   \draw [thick,domain=47:348] plot ({-1.5+cos(\x)}, {sin(\x)});
   
    
     \draw [black,thick,domain=0:86] plot ({-1+cos(\x)}, {-1+sin(\x)});
     \draw [black,thick,domain=96:165] plot ({-1+cos(\x)}, {-1+sin(\x)});

   \draw [black,thick,domain=177:360] plot ({-1+cos(\x)}, {-1+sin(\x)});
\end{tikzpicture}
 \caption{{\em Borromean rings}.}
 \label{fig_3003}

\end{figure}

%
%
%
%

%
%
%
%


\item[(3)]\, Here is the example of the {\em Trefoil knot}:
\begin{figure}[H]
 \begin{subfigure}[b]{0.49\textwidth}
    \centering
    \begin{tikzpicture}[scale=0.25]
        \draw (-6,-6) grid (6,6);

        \draw[line width=0.4mm,green] (-3,-6) -- (-3,6);
        \draw[line width=0.4mm,green] (3,-6) -- (3,6);
        \draw[line width=0.4mm,black] (-6,-4) -- (6,-4);
        \draw[line width=0.4mm,black] (-6,0) -- (6,0);
        \draw[line width=0.4mm,black] (-6,4) -- (6,4);

        \draw[line width=0.7mm,red]  (-6,-4) .. controls (-3.5,-4.5) .. (-3,-6);
        \draw[line width=0.7mm,red] (0,-4) .. controls (-2.5,-3.5) .. (-3,-2);

        \draw[line width=0.7mm,red]  (0,-4) .. controls (2.5,-4.5) .. (3,-6);
        \draw[line width=0.7mm,red] (6,-4) .. controls (3.5,-3.5) .. (3,-2);

        \draw[line width=0.7mm,red]  (-6,0) .. controls (-3.5,-0.5) .. (-3,-2);
        \draw[line width=0.7mm,red] (0,0) .. controls (-2.5,0.5) .. (-3,2);

        \draw[line width=0.7mm,red]  (0,0) .. controls (2.5,-0.5) .. (3,-2);
        \draw[line width=0.7mm,red] (6,0) .. controls (3.5,0.5) .. (3,2);

        \draw[line width=0.7mm,red]  (-6,4) .. controls (-3.5,3.5) .. (-3,2);
        \draw[line width=0.7mm,red] (0,4) .. controls (-2.5,4.5) .. (-3,6);

        \draw[line width=0.7mm,red]  (0,4) .. controls (2.5,3.5) .. (3,2);
        \draw[line width=0.7mm,red] (6,4) .. controls (3.5,4.5) .. (3,6);

        \shade[top color=yellow, bottom color=black] (-3,-4) --  (-6,-4) .. controls (-3.5,-4.5) .. (-3,-6) -- cycle;
        \shade[top color=yellow, bottom color=black] (-3,-4) -- (0,-4) .. controls (-2.5,-3.5) .. (-3,-2)-- cycle;

        \shade[top color=yellow, bottom color=black] (3,-4) --  (0,-4) .. controls (2.5,-4.5) .. (3,-6) -- cycle;
        \shade[top color=yellow, bottom color=black] (3,-4) -- (6,-4) .. controls (3.5,-3.5) .. (3,-2)-- cycle;

        \shade[top color=yellow, bottom color=black] (-3,0) --  (-6,0) .. controls (-3.5,-0.5) .. (-3,-2) -- cycle;
        \shade[top color=yellow, bottom color=black] (-3,0) -- (0,0) .. controls (-2.5,0.5) .. (-3,2)-- cycle;

        \shade[top color=yellow, bottom color=black] (3,0) --  (0,0) .. controls (2.5,-0.5) .. (3,-2) -- cycle;
        \shade[top color=yellow, bottom color=black] (3,0) -- (6,0) .. controls (3.5,0.5) .. (3,2)-- cycle;
        \shade[top color=yellow, bottom color=black] (-3,4) --  (-6,4) .. controls (-3.5,3.5) .. (-3,2) -- cycle;
        \shade[top color=yellow, bottom color=black] (-3,4) -- (0,4) .. controls (-2.5,4.5) .. (-3,6)-- cycle;

        \shade[top color=yellow, bottom color=black] (3,4) --  (0,4) .. controls (2.5,3.5) .. (3,2) -- cycle;
        \shade[top color=yellow, bottom color=black] (3,4) -- (6,4) .. controls (3.5,4.5) .. (3,6)-- cycle;
    \end{tikzpicture}
    \caption{$0<|\tau|<1$.}
    \label{fig_abcda}
  \end{subfigure}
 \begin{subfigure}[b]{0.49\textwidth}
    \centering
    \begin{tikzpicture}[scale=0.25]
        \draw (-6,-6) grid (6,6);

        \draw[line width=0.4mm,green] (-3,-6) -- (-3,6);
        \draw[line width=0.4mm,green] (3,-6) -- (3,6);
        \draw[line width=0.4mm,black] (-6,-4) -- (6,-4);
        \draw[line width=0.4mm,black] (-6,0) -- (6,0);
        \draw[line width=0.4mm,black] (-6,4) -- (6,4);

        \draw[line width=0.7mm,red]  (-3,-6) .. controls (-2.5,-4.5) .. (0,-4);
        \draw[line width=0.7mm,red] (-3,-2) .. controls (-3.5,-3.5) .. (-6,-4);

        \draw[line width=0.7mm,red]  (-3,-2) .. controls (-2.5,-0.5) .. (0,0);
        \draw[line width=0.7mm,red] (-3,2) .. controls (-3.5,0.5) .. (-6,0);

        \draw[line width=0.7mm,red]  (-3,2) .. controls (-2.5,3.5) .. (0,4);
        \draw[line width=0.7mm,red] (-3,6) .. controls (-3.5,4.5) .. (-6,4);

        \draw[line width=0.7mm,red]  (3,-6) .. controls (3.5,-4.5) .. (6,-4);
        \draw[line width=0.7mm,red] (3,-2) .. controls (2.5,-3.5) .. (0,-4);

        \draw[line width=0.7mm,red]  (3,-2) .. controls (3.5,-0.5) .. (6,0);
        \draw[line width=0.7mm,red] (3,2) .. controls (2.5,0.5) .. (0,0);

        \draw[line width=0.7mm,red]  (3,2) .. controls (3.5,3.5) .. (6,4);
        \draw[line width=0.7mm,red] (3,6) .. controls (2.5,4.5) .. (0,4);

        \shade[top color=yellow, bottom color=black] (-3,-4) -- (-3,-6) .. controls (-2.5,-4.5) .. (0,-4) -- cycle;
        \shade[top color=yellow, bottom color=black] (-3,-4) -- (-3,-2) .. controls (-3.5,-3.5) .. (-6,-4) -- cycle;

        \shade[top color=yellow, bottom color=black] (-3,0) --  (-3,-2) .. controls (-2.5,-0.5) .. (0,0) -- cycle;
        \shade[top color=yellow, bottom color=black] (-3,0) -- (-3,2) .. controls (-3.5,0.5) .. (-6,0) -- cycle;

        \shade[top color=yellow, bottom color=black] (-3,4) --  (-3,2) .. controls (-2.5,3.5) .. (0,4) -- cycle;
        \shade[top color=yellow, bottom color=black] (-3,4) -- (-3,6) .. controls (-3.5,4.5) .. (-6,4) -- cycle;

        \shade[top color=yellow, bottom color=black] (3,-4) -- (3,-6) .. controls (3.5,-4.5) .. (6,-4) -- cycle;
        \shade[top color=yellow, bottom color=black] (3,-4) -- (3,-2) .. controls (2.5,-3.5) .. (0,-4) -- cycle;

        \shade[top color=yellow, bottom color=black] (3,0) --  (3,-2) .. controls (3.5,-0.5) .. (6,0) -- cycle;
        \shade[top color=yellow, bottom color=black] (3,0) -- (3,2) .. controls (2.5,0.5) .. (0,0) -- cycle;

        \shade[top color=yellow, bottom color=black] (3,4) --  (3,2) .. controls (3.5,3.5) .. (6,4) -- cycle;
        \shade[top color=yellow, bottom color=black] (3,4) -- (3,6) .. controls (2.5,4.5) .. (0,4) -- cycle;
     \end{tikzpicture}
    \caption{$|\tau|>1$.}
    \label{fig_2003b}
  \end{subfigure}
  %
  %
%
  %
  %
 \caption{Coamoeba of the curve with defining polynomial $f(z,w)=1+\tau z^2+\tau w^3 + z^2w^3$ with (A) $0<|\tau|<1$ and (B) $|\tau|>1$  ({\em Trefoil knots}).}
 \label{fig_abcd}
\end{figure}

\begin{figure}[H]
\begin{subfigure}[b]{0.4\textwidth}
 \centering 
 \includegraphics[scale=0.25]{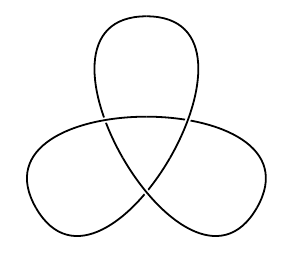}
 \caption{}
 \label{fig_trefoil1}
 \end{subfigure}
 \begin{subfigure}[b]{0.4\textwidth}
  \centering
  \includegraphics[scale=0.25]{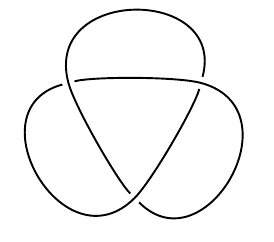}
 \caption{}
 \label{fig_trefoil2}
\end{subfigure}
\caption{Trefoil knots}
\label{fig_knot_trefoil}
\end{figure}

\end{itemize}
\end{example}


\begin{thebibliography}{9}   







\bibitem[GKZ-94]{GKZ-94}{\sc I. M. Gelfand, M.
M. Kapranov and A. V. Zelevinski}, {\em Discriminants, resultants and
multidimensional determinants},
Birkh{\"a}user Boston 1994.













\bibitem[MN-15]{MN-15} {\sc F.~Madani and M.~Nisse}, {\em  Analytic varieties with finite volume amoebas are algebraic}, J. Reine Angew. Math., Vol. {\bf 706}, (2015), 67–81.




\bibitem[M1-00]{M1-00}{\sc G. Mikhalkin }, {\em  Real algebraic
curves, moment map and amoebas}, Ann. of Math. {\bf 151} (2000), 309-326.




 








\bibitem[NS1-13]{NS1-13}{\sc M. Nisse  and F. Sottile}, {\em Non-Archimedean coAmoebas}, Contemporary Mathematics 
Amer. Math. Soc. Vol. {\bf 605}, (2013), 73--91.

  
 

 
 





\bibitem[PR-11]{PR-11}{\sc M. Passare  and J. J. Risler}, {\em On the curvature of the Real Amoeba}, Proceedings of 17th Gokova Geometry-Topology Conference, (2011),  pp. 129-134.

\bibitem[PT-05]{PT-05}{\sc M. Passare  and A. Tsikh}, {\em Amoebas: their spines and their contours}, Contemporary Mathematics 377 (2005) 275 – 288.













 

 
\end{thebibliography}
\end{document}